\documentclass[final,12pt]{elsarticle}

\usepackage[utf8]{inputenc}
\usepackage[T1]{fontenc}
\usepackage[english]{babel}
\usepackage{amssymb,mathtools}
\usepackage{booktabs}
\usepackage{array}
\usepackage{caption}
\usepackage{amsthm}
\usepackage{placeins}
\usepackage{microtype}
\newtheorem{theorem}{Theorem}[section]
\newtheorem{proposition}[theorem]{Proposition}
\newtheorem{lemma}[theorem]{Lemma}

\theoremstyle{definition}

\theoremstyle{remark}

\begin{document}

\begin{frontmatter}

\title{Same Mean Lag, Different Stability:\\
Periodic Implementation in Adaptive Competition}

\author{Alexander Omelchenko}
\ead{aomelchenko@constructor.university}
\address{Constructor University Bremen gGmbH,
Campus Ring 1, 28759 Bremen, Germany}

\begin{abstract}
Can two implementation processes with the same mean lag give opposite stability
predictions for adaptive competition? We show that they can, and quantify the
discrepancy. The motivating problem is technological coadaptation: decisions are
revised while earlier changes are still being implemented and the competing
system continues to change. Configuration updates in adaptive cyber defence and
activation of revised algorithmic trading strategies provide concrete settings
for this modelling question. A coupled replicator--mutator model separates
intended portfolios from those actually in use. Near a uniform interior
equilibrium, the interaction geometry determines the feedback branches; under
strict antagonism, the largest singular value identifies the first unstable
mode. For continuous implementation, hard observation and implementation lags
enter through one delay budget. Replacing one smooth implementation filter by
periodic partial updates, with zero hard lags and a uniform eligibility-phase
calibration, preserves the mean waiting time but lowers the first gain boundary
by an explicit positive quadratic correction for short release periods. A
reproducible three-strategy benchmark exhibits decay under the continuous model
and growth under periodic implementation at the same mean lag and interaction
gain. Balanced performance observables can also understate the underlying
portfolio oscillations. The results identify a concrete failure of mean-only
timing descriptions and provide a benchmark for modelling adaptation and
implementation separately.
\end{abstract}


\end{frontmatter}

\section{Introduction}
\label{sec:intro}

Adaptive technological competition provides a setting in which revising a
strategy portfolio and implementing it are distinct processes. Consider two
populations that reallocate effort among alternative configurations according to
their relative performance against the other population's implemented mix. A
revision changes the intended portfolio first. Validation, integration, or
replacement of existing configurations may delay its implementation, while the
other population continues to adapt. A decision based on an observed portfolio
may therefore take effect only after that portfolio has changed.

Concrete settings for this question include adaptive security configurations
and the activation of revised trading strategies. Coevolutionary models have
been studied for cybersecurity \cite{bashir2026cyber} and for interacting
automated traders \cite{cliff2022trading}; software release engineering also
distinguishes completed revisions from changes introduced into running systems
\cite{warner2018canary}. These settings motivate a common comparison: keep the
interaction and revision dynamics fixed, and change only how intended changes
become active. Section~\ref{sec:applications} specifies the quantities and
questions through which the benchmark can be transferred.

The implementation rule matters as well as its average speed. Revisions may be
introduced continuously through gradual adjustment, or through scheduled partial
updates. We ask whether matching the mean eligibility-to-deployment waiting time
under a stated phase convention makes these descriptions interchangeable for
local stability analysis. We also ask which paired changes in the two portfolios
first become unstable and whether aggregate performance reveals the resulting
compositional motion. These questions motivate the separation of revision and
deployment states in the coupled model below: both populations adapt to one
another, rather than one merely tracking an externally prescribed signal.

Our main finding is a \emph{stability-disagreement interval}: the continuous
model predicts decay of perturbations for interaction gains at which the
mean-matched periodic system is already unstable. This interval is quantified
by the first-boundary expansion in Theorem~\ref{thm:cal-critical-shift} and
illustrated by a same-gain nonlinear comparison in
Figure~\ref{fig:matched-mean-contrast}. The contribution is therefore not only
the existence of a bifurcation, but a testable distinction between two
implementation models that a mean timing statistic cannot separate.

\begin{figure}[!t]
\centering
\includegraphics[width=0.96\textwidth]{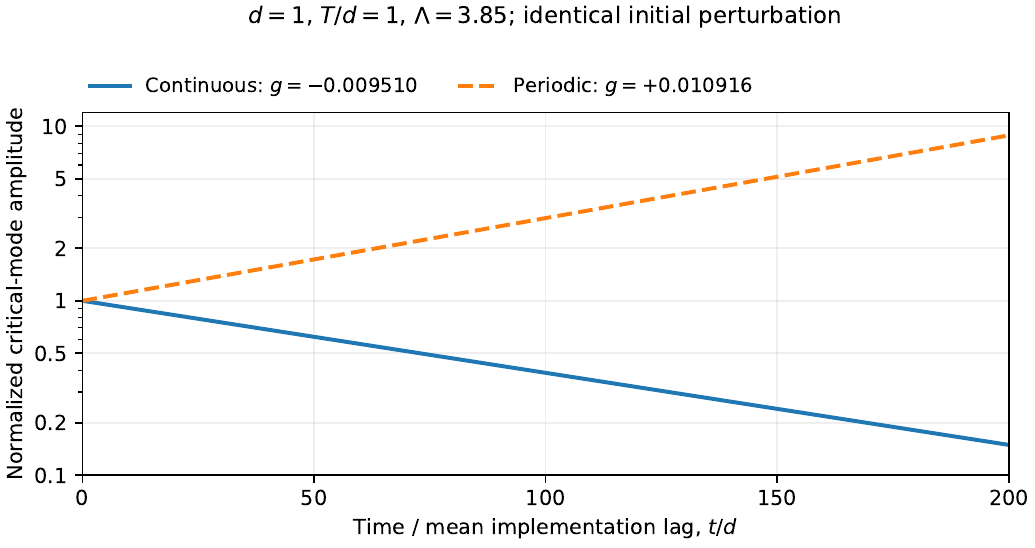}
\caption{Same mean lag, same gain, opposite local stability. Both nonlinear
systems start from the same size-$10^{-4}$ portfolio perturbation with
$d=1$, $\Lambda=3.85$, and zero hard lags. The periodic system uses $T/d=1$
and $\varrho_T=2/3$. The linear growth rates are $-0.009509796$ and
$0.010915576$ per unit time; the nonlinear fitted rates are $-0.009509803$
and $0.010915483$. Amplitudes use the respective left critical-mode
projections, normalized initially, sampled after releases, and shown on a
logarithmic vertical axis. Section~\ref{sec:num} gives the comparison and checks.}
\label{fig:matched-mean-contrast}
\end{figure}
\FloatBarrier

Replicator dynamics provide a standard description of frequency-dependent
selection, while replicator--mutator equations incorporate mutation, persistent
exploration, or switching \cite{taylor1978,hofbauer1998,page2002}. We use the
parent-independent uniform-mutation form, which also describes exploration
toward the barycenter. Delayed replicator systems have been studied with delayed
payoff or state information, distributed delays, mutation feedback, and
multiple delays across interacting populations
\cite{tao1997,wesson2016,benkhalifa2017,mittal2020,twocommunity2020,wettergren2023,wang2024mutation,sdgames2025}. Related evolutionary-game models
have been used for delayed decision-making and coevolutionary dynamics
in cybersecurity \cite{bashir2026cyber,tan2025honeypoint}. The present model
focuses on a different timing question: the state selected by revision need not
be the state currently deployed.

A first-order deployment law is a compartmental realization of an exponential
distributed-delay kernel. Related constructions occur in compartment and
structured-replicator models \cite{ficcompartment2025,miekiszstructured2025},
and cyclic compartment systems admit general distributed-delay representations
\cite{cassidy2021distributed}. After linearization, the continuous model below
produces cyclic negative-feedback branches. Their stability is governed by the
secant and delayed-feedback framework
\cite{arcak2006secant,ahsen2015secant}. Componentwise time shifts explain why
the discrete lags around a single cycle enter through their sum
\cite{lucken2013reduction,lucken2015classification}; in the present variables,
that cycle sum is the hard-delay budget.

Scheduled implementation gives a hybrid problem rather than another continuous
filter. Periodic impulses have been studied in evolutionary games, including
bimatrix replicator systems
\cite{wang2008impulses,zhang2010impulses,an2020impulses}, and Poincar\'e-map
methods are standard for impulsive biological feedback models
\cite{churilov2016goodwin}. Small-sampling continuations from continuous Hopf
boundaries also appear in sampled-data feedback theory
\cite{su2016sampled,su2021sampled,hetel2017sampling}, while same-mean delay
comparisons show that temporal distribution can affect a stability boundary
\cite{atay2008delayed}. Here the reset acts only on the deployed portfolio. The
revision portfolio continues to evolve between releases, and each release moves
the deployed state part of the way toward current intent.

The construction extends an earlier one-population simplex implementation model
\cite{omelchenko2026voteshare} to two coupled populations. Under the
equilibrium-balance conditions, the uniform state is an equilibrium and the
tangent characteristic equation factorizes as
\begin{equation}
\det\!\left[R(z)I-Ke^{-z\tau_\Sigma}\mathcal A\mathcal B\right]=0,
\qquad
R(z)=(z+\mu_X)(z+\mu_Y)(z+\kappa_X)(z+\kappa_Y).
\label{eq:intro-spectral-reduction}
\end{equation}
Thus the interaction geometry enters through the projected product
\(\mathcal A\mathcal B\). Under strict antagonism, its nonzero eigenvalues are
\(-\chi s_j^2\), where \(s_j\) are the singular values of \(\mathcal A\). The
singular vectors identify paired compositional modes, and the singular values
order the associated feedback strengths. The continuous branch calculation
then gives the zero-delay and Hopf boundaries in the original model parameters
and identifies the first mode to lose stability.

For periodic implementation, the hard discrete lags are set to zero and one
smooth filter is replaced by a common partial-release rule. The singular-mode
decomposition persists. Under the uniform eligibility-phase calibration, the
release fraction is chosen so that the mean waiting time from eligibility to
deployment equals that of the continuous filter. For sufficiently short release
periods, the first periodic boundary satisfies
\begin{equation}
\Lambda_{\rm cal}(T)=\Lambda_0-\mathfrak c T^2+O(T^4),
\qquad \mathfrak c>0.
\label{eq:intro-calendar-shift}
\end{equation}
The associated phase-averaged rollout kernel has the same leading displacement,
although its finite-period dynamics differ from those of the synchronized
hybrid system. This comparison isolates an implementation effect that is
not captured by the matched mean. The remaining results interpret the delay
budget and relate hidden compositional motion to scalar performance. A
three-strategy example tests the continuous and periodic predictions.

Section~\ref{sec:model} introduces the model and its spectral reduction.
Sections~\ref{sec:lin} and~\ref{sec:calendarized-fielding} treat continuous and
periodic deployment. Section~\ref{sec:consequences} interprets timing and
observability, Section~\ref{sec:num} provides the numerical
illustrations, and Section~\ref{sec:disc} concludes. Technical arguments and
computational details are collected in the appendices.
\section{Model and spectral reduction}
\label{sec:model}

A strategy represents a configuration or method in a fixed repertoire, and a
portfolio records relative allocation across that repertoire. Revision changes
the intended mix; implementation changes the mix actually participating in the
interaction. This distinguishes adaptation of decisions from completed changes
in use, without introducing a model of absolute inventories. Let
\[
\begin{aligned}
\Delta_m&=\{z\in\mathbb R_{\ge0}^m:\mathbf 1^\top z=1\},
&u&=\frac1m\mathbf 1,\\
T\Delta_m&=\{z\in\mathbb R^m:\mathbf 1^\top z=0\},
&P&=I-\frac1m\mathbf 1\mathbf 1^\top.
\end{aligned}
\]
Population \(X\) has an intended (revision) portfolio \(p(t)\in\Delta_m\) and
an implemented (deployed) portfolio \(x(t)\in\Delta_m\); population \(Y\) has
corresponding portfolios \(q(t),y(t)\in\Delta_m\). For interaction matrices
\(A,B\in\mathbb R^{m\times m}\), the model is
\begin{subequations}
\begin{align}
\dot p
&=\lambda_Xp\odot\Bigl[Ay(t-\sigma_X)
-p(t)^\top Ay(t-\sigma_X)\mathbf 1\Bigr]+\mu_X(u-p),
\label{eq:p-dynamics}\\
\dot q
&=\lambda_Yq\odot\Bigl[Bx(t-\sigma_Y)
-q(t)^\top Bx(t-\sigma_Y)\mathbf 1\Bigr]+\mu_Y(u-q),
\label{eq:q-dynamics}\\
\dot x&=\kappa_X\bigl[p(t-\theta_X)-x(t)\bigr],
\label{eq:x-deployment}\\
\dot y&=\kappa_Y\bigl[q(t-\theta_Y)-y(t)\bigr].
\label{eq:y-deployment}
\end{align}
\end{subequations}
Here \(\sigma_X,\sigma_Y\ge0\) are observation lags,
\(\theta_X,\theta_Y\ge0\) are hard implementation lags,
\(\lambda_X,\lambda_Y>0\) are selection rates,
\(\kappa_X,\kappa_Y>0\) are implementation rates, and
\(\mu_X,\mu_Y>0\) are uniform exploration rates. The mutation/exploration
terms can be written as
\(\mu_X(u\mathbf 1^\top-I)p\) and
\(\mu_Y(u\mathbf 1^\top-I)q\); thus they are parent-independent uniform
Markov mutation terms. Throughout, ``replicator--mutator'' refers to this
specific form: exploration redistributes weight among existing strategies rather
than creating new ones. The parameters \(\lambda_X,\lambda_Y\) measure revision
responsiveness, whereas \(\kappa_X,\kappa_Y\) govern how quickly implemented
portfolios approach delayed intent. The current normalization in the revision
equations preserves
the affine constraints, and \(\Delta_m^4\) is forward invariant for
continuous initial histories. The proof and the shifted-exponential
representation of the deployment filters are given in
Appendix~\ref{app:sec2}.

To analyse adjustment near a mixed reference configuration, assume the
equilibrium-balance conditions
\begin{equation}
Au=a\mathbf 1,\qquad Bu=b\mathbf 1.
\label{eq:equilibrium-balance}
\end{equation}
They ensure that \((p,q,x,y)=(u,u,u,u)\) is an equilibrium for all lags.
No left balance condition is needed for the equilibrium or the spectral
reduction. Define
\begin{equation}
\mathcal A=\left.PAP\right|_{T\Delta_m},
\qquad
\mathcal B=\left.PBP\right|_{T\Delta_m},
\qquad
\mathcal C=\mathcal A\mathcal B.
\label{eq:projected-matrices}
\end{equation}
Linearizing at this equilibrium gives
\begin{subequations}
\begin{align}
\dot\alpha
&=-\mu_X\alpha+\frac{\lambda_X}{m}\mathcal A\eta(t-\sigma_X),
&
\dot\beta
&=-\mu_Y\beta+\frac{\lambda_Y}{m}\mathcal B\xi(t-\sigma_Y),
\label{eq:revision-linear}\\
\dot\xi
&=\kappa_X\bigl[\alpha(t-\theta_X)-\xi\bigr],
&
\dot\eta
&=\kappa_Y\bigl[\beta(t-\theta_Y)-\eta\bigr].
\label{eq:deployment-linear}
\end{align}
\end{subequations}

\begin{proposition}[characteristic reduction]
\label{prop:characteristic-reduction}
The characteristic equation of the tangent system is
\begin{equation}
\det\!\left[R(z)I-Ke^{-z\tau_\Sigma}\mathcal C\right]=0,
\label{eq:uniform-determinant-factor}
\end{equation}
where
\begin{equation}
R(z)=(z+\mu_X)(z+\mu_Y)(z+\kappa_X)(z+\kappa_Y),
\qquad
K=\frac{\lambda_X\lambda_Y\kappa_X\kappa_Y}{m^2},
\label{eq:R-and-K}
\end{equation}
and
\begin{equation}
\tau_\Sigma=\sigma_X+\sigma_Y+\theta_X+\theta_Y.
\label{eq:total-delay}
\end{equation}
Hence every \(\gamma\in\sigma(\mathcal C)\) gives the scalar branch
\begin{equation}
R(z)=K\gamma e^{-z\tau_\Sigma}.
\label{eq:deployment-characteristic}
\end{equation}
\end{proposition}

The derivation is given in Appendix~\ref{app:sec2}.
Equation~\eqref{eq:deployment-characteristic} separates two sources of delayed
response. Hard waits enter through their total \(\tau_\Sigma\), whereas gradual
implementation contributes the filter factors \(z+\kappa_X\),
\(z+\kappa_Y\) and also enters \(K\). Removing a hard wait is therefore not the
same parameter change as speeding up implementation.
Section~\ref{sec:lin} uses this separation to distinguish delay-induced
instability from instability already present at zero hard delay.
\section{Continuous deployment: stability and first Hopf crossing}
\label{sec:lin}

For an explicit reference case of opposing relative-performance feedback,
assume strict antagonism on the tangent space,
\begin{equation}
\mathcal B=-\chi\mathcal A^\top,
\qquad \chi>0.
\label{eq:antagonistic}
\end{equation}
If \(s_1\ge\cdots\ge s_r>0\) are the nonzero singular values of
\(\mathcal A\), then
\(\mathcal C=-\chi\mathcal A\mathcal A^\top\) has eigenvalues
\(\gamma_j=-\chi s_j^2\). Hence each singular pair defines a coupled
compositional mode with branch strength
\begin{equation}
L_j=K\chi s_j^2.
\label{eq:singular-branch-strength}
\end{equation}
A mode is a paired change in the two portfolio compositions, not a single
strategy. At fixed interaction matrices and exploration and implementation
rates, increasing \(\lambda_X\lambda_Y\) raises every nonzero \(L_j\). Faster
revision thus strengthens the feedback; it is not equivalent to faster
implementation.
The implementation factors are first-order compartments, equivalently
exponential distributed-delay kernels \cite{cassidy2021distributed}. After the
spectral reduction in Proposition~\ref{prop:characteristic-reduction}, every
nonzero mode satisfies
\begin{equation}
\begin{aligned}
R(z)+Le^{-z\tau}&=0,
& R(z)&=\prod_{k=1}^4(z+r_k),\\
(r_1,r_2,r_3,r_4)&=(\mu_X,\mu_Y,\kappa_X,\kappa_Y),
& \tau&=\tau_\Sigma.
\end{aligned}
\label{eq:negative-branch-characteristic}
\end{equation}
with \(L=L_j\). This is a four-stage cyclic negative-feedback branch. Its
secant-type stability theory and delayed extension are classical
\cite{arcak2006secant,ahsen2015secant}; here the poles and loop gain are fixed
by the replicator--mutator parameters. The dependence on \(\tau_\Sigma\) is the
single-cycle instance of componentwise time-shift invariance in delay networks
\cite{lucken2013reduction,lucken2015classification}.

Let \(c_3,c_2,c_1,c_\ast\) be the elementary symmetric sums in
\(r_1,\ldots,r_4\), and set
\begin{equation}
H=\frac{(c_3c_2-c_1)c_1}{c_3^2}.
\label{eq:H-definition}
\end{equation}
At zero hard delay the exact Routh--Hurwitz condition is
\begin{equation}
L<H-c_\ast.
\label{eq:negative-branch-zero-delay-margin}
\end{equation}
At the boundary, the imaginary pair satisfies
\(\omega_0^2=c_1/c_3\). Moreover,
\begin{equation}
H-c_\ast\ge4c_\ast,
\label{eq:H-five-cstar}
\end{equation}
with equality for equal rates, recovering the classical four-stage secant
constant.

A positive-frequency crossing exists exactly when \(L>c_\ast\). In the
zero-delay-stable range \(c_\ast<L<H-c_\ast\), its frequency and first critical
delay are determined by
\begin{equation}
\prod_{k=1}^4(r_k^2+\omega_\ast^2)=L^2,
\label{eq:omega-star-equation}
\end{equation}
\begin{equation}
\tau_{\rm first}(L)
=\frac{\pi-\phi(\omega_\ast)}{\omega_\ast},
\qquad
\phi(\omega)=\sum_{k=1}^4\arctan\frac{\omega}{r_k}.
\label{eq:first-delay-branch}
\end{equation}
The frequency is unique, every crossing is simple and directed from left to
right, and \(\tau_{\rm first}(L)\) decreases strictly with \(L\); the details
are collected in Appendix~\ref{app:sec4}.

\begin{theorem}[Continuous stability boundary and critical interaction mode]
\label{thm:delay-induced-hopf}
Let \(L_{\max}=K\chi s_1^2\). The barycentric equilibrium has the following
continuous-deployment classification:
\begin{equation}
\begin{array}{ll}
L_{\max}\le c_\ast
&:\ \text{asymptotically stable for every finite }\tau_\Sigma\ge0,\\[0.2em]
c_\ast<L_{\max}<H-c_\ast
&:\ \text{first loss of stability at }
\tau_{\rm crit}=\tau_{\rm first}(L_{\max}),\\[0.2em]
L_{\max}\ge H-c_\ast
&:\ \text{at or beyond the zero-delay stability boundary}.
\end{array}
\label{eq:delay-induced-window}
\end{equation}
In the middle case,
\begin{equation}
\tau_{\rm crit}
=\min_{j:\,L_j>c_\ast}\tau_{\rm first}(L_j)
=\tau_{\rm first}(L_{\max}).
\label{eq:global-critical-delay-antagonistic}
\end{equation}
If \(s_1\) is simple, its singular mode is the unique first mode to lose
stability, and a simple conjugate pair crosses the imaginary axis at
\(\tau_{\rm crit}\). A nonzero first Lyapunov coefficient gives a nondegenerate
Hopf bifurcation and determines the direction and orbital stability of the local
periodic branch.
\end{theorem}

If the largest singular value has multiplicity \(k>1\), its full
\(k\)-dimensional singular eigenspace reaches the boundary simultaneously and
the real center subspace has dimension \(2k\). The simple-Hopf conclusion and
the one-mode normal-form calculation in Appendix~\ref{app:baseline-lyapunov}
then do not apply. The baseline example has a simple critical mode and a
negative cubic Hopf coefficient; see Section~\ref{sec:num} and
Appendix~\ref{app:baseline-lyapunov}. Section~\ref{sec:calendarized-fielding}
next asks how the zero-hard-delay interaction-gain boundary changes when one
smooth implementation filter is replaced by a periodic release rule.

\section{Mean-matched periodic deployment}
\label{sec:calendarized-fielding}

The continuous implementation law in Section~\ref{sec:model} represents smooth
catch-up of the deployed portfolio. Scheduled release windows instead hold the
implemented state between release times and update it in batches. Periodic impulses
have been studied in evolutionary games, including bimatrix replicator dynamics
\cite{wang2008impulses,zhang2010impulses,an2020impulses}, while impulse-to-impulse maps are a
standard tool for hybrid biological feedback systems \cite{churilov2016goodwin}.
Here the impulse acts on a separate deployed portfolio rather than directly on
population shares or on the revision dynamics. We compare this mechanism with
the continuous filter after calibrating both descriptions to the same mean
waiting time from eligibility to deployment under a stated phase convention;
related comparisons of feedback delays with a common mean are discussed in
\cite{atay2008delayed}.

To isolate deployment geometry, set the four hard lags to zero and replace only
the \(X\)-side implementation equation by the simplex-preserving rule
\begin{equation}
 \dot x(t)=0\quad(t\ne t_n),
 \qquad
 x(t_n^+)=(1-\varrho)x(t_n^-)+\varrho p(t_n^-),
 \qquad t_n=nT,
\label{eq:cal-nonlinear-reset}
\end{equation}
where \(0<\varrho\le1\). Each update closes a fraction of the gap to current
intent, while revision continues between updates. The reset does not represent
delivery of an earlier order with previously fixed contents. The revision
equations and the \(Y\)-side continuous filter remain unchanged. Linearization
at the barycentric equilibrium gives
\begin{equation}
\begin{aligned}
 \dot\alpha&=-\mu_X\alpha+\frac{\lambda_X}{m}\mathcal A\eta,
 &\dot\beta&=-\mu_Y\beta+\frac{\lambda_Y}{m}\mathcal B\xi,\\
 \dot\xi&=0,
 &\dot\eta&=\kappa_Y(\beta-\eta),
 \qquad t\ne t_n,
\end{aligned}
\label{eq:cal-linear-flow}
\end{equation}
with \(\xi(t_n^+)=(1-\varrho)\xi(t_n^-)+\varrho\alpha(t_n^-)\).

\begin{proposition}[Modal monodromy reduction]
\label{prop:cal-modal-reduction}
Assume \(\mathcal B=-\chi\mathcal A^\top\), and let
\(\mathcal A v_j=s_j u_j\), \(\mathcal A^\top u_j=s_jv_j\), \(s_j>0\), be a
singular pair. Then the corresponding four-dimensional singular-mode subspace
is invariant under both the flow and the reset. After a constant diagonal
similarity, its dynamics are
\begin{equation}
 \dot w=A_{\Lambda_j}w,
 \qquad
 w(t_n^+)=J_{\varrho}w(t_n^-),
 \qquad
 M_{\Lambda_j}(T)=J_{\varrho}e^{A_{\Lambda_j}T},
\label{eq:cal-modal-monodromy}
\end{equation}
where
\begin{equation}
 \Lambda_j=\frac{\chi\lambda_X\lambda_Y}{m^2}s_j^2.
\label{eq:cal-modal-gain}
\end{equation}
For the continuous deployment law of Section~\ref{sec:lin}, the corresponding
branch strength is \(L_j=\kappa_X\kappa_Y\Lambda_j\). Here
\begin{equation}
 A_\Lambda=
 \begin{pmatrix}
 -\mu_X&0&0&\sqrt\Lambda\\
 0&-\mu_Y&-\sqrt\Lambda&0\\
 0&0&0&0\\
 0&\kappa_Y&0&-\kappa_Y
 \end{pmatrix},
 \qquad
 J_{\varrho}=
 \begin{pmatrix}
 1&0&0&0\\
 0&1&0&0\\
 \varrho&0&1-\varrho&0\\
 0&0&0&1
 \end{pmatrix}.
\label{eq:cal-branch-matrices}
\end{equation}
Thus a common scalar release rule preserves the interaction modes obtained in
Section~\ref{sec:model}. More generally, a tangent reset
\(\xi^+=(I-\mathcal R)\xi^-+\mathcal R\alpha^-\) preserves every spectral
eigenspace of \(\mathcal A\mathcal A^\top\) exactly when
\begin{equation}
 [\mathcal R,\mathcal A\mathcal A^\top]=0.
\label{eq:cal-commuting-reset}
\end{equation}
For a repeated singular value, \(\mathcal R\) may mix vectors within the
corresponding eigenspace. This commutator condition concerns the tangent
linearization; a nonlinear reset on \(\Delta_m\) additionally requires a
simplex-preserving lift, and commutation alone does not ensure nonnegativity or
conservation of total mass.
\end{proposition}

Put \(d=\kappa_X^{-1}\). The impulse response of the replaced
first-order filter is exponential, with mean waiting time \(d\) and variance
\(d^2\). Although the release rule \eqref{eq:cal-nonlinear-reset} is
synchronized, calibrating its release fraction against \(d\) requires a phase
convention. We use the following uniform-phase, geometric-clearing convention.
For an infinitesimal revision that has become eligible for deployment, let
\(U_T\sim\operatorname{Unif}(0,T)\) be the time to the next release window, and
let \(N_{\rm add}\) be independent of \(U_T\), with
\[
 \mathbb P(N_{\rm add}=n)=\varrho(1-\varrho)^n,
 \qquad n=0,1,\ldots,
\]
so that \(N_{\rm add}\) counts additional windows missed before deployment.
The eligibility-to-deployment waiting time
\(W_T=U_T+TN_{\rm add}\) satisfies
\begin{equation}
\begin{aligned}
 \mathbb E W_T
 &=\frac{T}{2}+T\frac{1-\varrho}{\varrho},\\
 \operatorname{Var}W_T
 &=\frac{T^2}{12}+T^2\frac{1-\varrho}{\varrho^2}.
\end{aligned}
\label{eq:cal-waiting-time-moments-general}
\end{equation}
Matching \(\mathbb E W_T=d\) gives
\begin{equation}
 \varrho_T=\frac{T}{d+T/2}
 =\frac{\kappa_XT}{1+\kappa_XT/2},
 \qquad 0<T<2d.
\label{eq:cal-mean-matched-reset}
\end{equation}
Under this calibration,
\begin{equation}
 \mathbb E W_T=d,
 \qquad
 \operatorname{Var}W_T=d^2-\frac{T^2}{6},
 \qquad
 \mathbb E W_T^2=2d^2-\frac{T^2}{6}.
\label{eq:cal-waiting-time-moments-matched}
\end{equation}
Thus the common mean conceals a systematic change in the waiting-time law:
the periodic variance is smaller by \(T^2/6\). This deficit determines the
change in the quadratic term of the transform's expansion near zero frequency,
but does not by itself determine the stability shift at a nonzero Hopf
frequency. Moment-based stability approximations and same-mean comparisons for
distributed delays provide a broader context
\cite{campbell2009moments,bernard2015optimal}; the explicit sign below follows
from the present hybrid branch. The calibration is phase averaged, while the
theorem concerns the synchronized hybrid system itself. The continuous
same-mean branch has characteristic function
\begin{equation}
 F(z,\Lambda)=D(z)+\kappa_Y\Lambda,
 \qquad
 D(z)=(z+\mu_X)(z+\mu_Y)(z+\kappa_Y)(1+dz).
\label{eq:cal-continuous-branch}
\end{equation}
Let \(\omega_0>0\) be determined by
\begin{equation}
 \arctan\frac{\omega_0}{\mu_X}
 +\arctan\frac{\omega_0}{\mu_Y}
 +\arctan\frac{\omega_0}{\kappa_Y}
 +\arctan(d\omega_0)=\pi,
 \qquad
 \Lambda_0=\frac{|D(i\omega_0)|}{\kappa_Y}.
\label{eq:cal-continuous-boundary}
\end{equation}
Since \(\kappa_X=d^{-1}\) and the continuous branch strength is
\(L=(\kappa_Y/d)\Lambda\), the zero-hard-delay boundary in
\eqref{eq:negative-branch-zero-delay-margin} gives the equivalent expression
\begin{equation}
 \Lambda_0=\frac{d}{\kappa_Y}(H-c_\ast).
\label{eq:cal-continuous-boundary-RH}
\end{equation}
Thus the continuous branch is stable for \(0\le\Lambda<\Lambda_0\), and at
\(\Lambda_0\) it has one simple conjugate pair \(\pm i\omega_0\) on the
imaginary axis while its remaining eigenvalues lie strictly in the left
half-plane.

\begin{theorem}[First periodic gain boundary and its shift]
\label{thm:cal-critical-shift}
For the reset fraction \eqref{eq:cal-mean-matched-reset}, there are \(T_0>0\)
and functions \(\Lambda_{\rm cal}(T)\) and \(\omega_{\rm cal}(T)\) on
\([0,T_0)\) that admit even real-analytic extensions through \(T=0\) and
continue the simple continuous boundary \((\omega_0,\Lambda_0)\). For every
\(0<T<T_0\),
\(\Lambda_{\rm cal}(T)\) is the first positive gain boundary of the periodic
branch: all Floquet multipliers lie in the open unit disk for
\(0\le\Lambda<\Lambda_{\rm cal}(T)\); at
\(\Lambda=\Lambda_{\rm cal}(T)\) there is exactly one simple unit-modulus pair
\(\exp(\pm i\omega_{\rm cal}(T)T)\); and the branch is unstable for \(\Lambda\)
immediately above this value. Moreover,
\begin{equation}
 \Lambda_{\rm cal}(T)=\Lambda_0-\mathfrak cT^2+O(T^4),
 \qquad
 \omega_{\rm cal}(T)=\omega_0+\nu_{\rm cal}T^2+O(T^4),
\label{eq:cal-boundary-expansions}
\end{equation}
where
\begin{equation}
\begin{aligned}
 a_0&=\sum_{r\in\{\mu_X,\mu_Y,\kappa_Y\}}
       \frac{r}{r^2+\omega_0^2}
       +\frac{d}{1+d^2\omega_0^2},\\
 b_0&=\sum_{r\in\{\mu_X,\mu_Y,\kappa_Y\}}
       \frac{1}{r^2+\omega_0^2}
       +\frac{d^2}{1+d^2\omega_0^2},
\end{aligned}
\label{eq:cal-a0-b0}
\end{equation}
\begin{equation}
 \nu_{\rm cal}=-\frac{d\omega_0^3}
 {12(1+d^2\omega_0^2)a_0}<0,
 \qquad
 \mathfrak c=
 \frac{\Lambda_0\omega_0^2}{12(1+d^2\omega_0^2)}
 \left(1+\frac{d\omega_0^2b_0}{a_0}\right)>0.
\label{eq:cal-coefficients}
\end{equation}
The critical Floquet pair crosses outwards as \(\Lambda\) increases. Hence
periodic partial deployment shifts the first positive interaction-gain boundary
downward, to second order in its release period, relative to the continuous
filter with the same calibrated mean eligibility-to-deployment waiting time.
\end{theorem}

The proof and the model-specific determinant expansion are given in
Appendix~\ref{app:calendarized-fielding}; standard monodromy and symmetric
composition machinery is cited rather than rederived. The shift gives a concrete
modelling discrepancy: for small \(T\), branches with gains just above
\(\Lambda_{\rm cal}(T)\) but below \(\Lambda_0\) are stable in the continuous
approximation and unstable under periodic implementation, despite the matched
mean waiting time.

The sign of the shift is not tied to the uniform-phase calibration. If
the reset instead retains exactly the continuous filter's one-period decay,
\(1-\varrho_T^{\rm pole}=e^{-T/d}\), then the same model-specific calculation, within the sampled-data framework
reviewed in \cite{hetel2017sampling}, gives
\begin{equation}
 \Lambda_{\rm pole}(T)
 =\Lambda_0-\frac{\Lambda_0\omega_0^2b_0}{12d\,a_0}\,T^2+O(T^4)
 <\Lambda_0
 \qquad (T>0\ \hbox{small}).
\label{eq:cal-pole-matched-shift}
\end{equation}
Thus the coefficient depends on the calibration convention, but the downward
quadratic shift persists. The accompanying frequency correction is recorded in
Appendix~\ref{app:calendarized-fielding}.

For unit rates, \(\mu_X=\mu_Y=\kappa_X=\kappa_Y=1\) (and hence \(d=1\)),
\begin{equation}
 \Lambda_{\rm cal}(T)=4-\frac13T^2+O(T^4),
 \qquad
 \omega_{\rm cal}(T)=1-\frac1{48}T^2+O(T^4).
\label{eq:cal-equal-rate-result}
\end{equation}
Because \(\Lambda_j\) is proportional to \(s_j^2\), a common increase of
interaction responsiveness reaches the calendarized boundary first in a mode
with largest singular value. If that value has multiplicity \(k>1\), its full
spectral eigenspace reaches the boundary simultaneously; no one-mode nonlinear
bifurcation classification is asserted in that case.

The waiting-time law \(W_T\) used for calibration also defines a
phase-averaged rollout kernel with transform
\begin{equation}
 \widehat H_T(z)=\frac{1-e^{-zT}}{zT}
 \frac{\varrho_T}{1-(1-\varrho_T)e^{-zT}}
 =\frac1{1+dz}-\frac{z^2T^2}{12(1+dz)^2}+O(T^4).
\label{eq:cal-kernel-expansion}
\end{equation}
The full correction is \(-z^2T^2/[12(1+dz)^2]\), not only its
low-frequency term \(-z^2T^2/12\). The boundary calculation evaluates this
frequency-dependent correction at \(z=i\omega_0\); higher waiting-time
moments also change at order \(T^2\), so variance alone does not determine
\(\mathfrak c\). This kernel nevertheless has the same leading displacement \eqref{eq:cal-boundary-expansions}
of the simple continuous boundary, although its finite-\(T\) dynamics are not
identical to the synchronized reset. Thus phase averaging does not recover the
continuous exponential benchmark at order \(T^2\): the leading discrepancy is
not solely an effect of simultaneous updates. The theorem does not include hard
discrete lags; their addition leads to a monodromy operator on the history space
rather than the finite-dimensional matrix \eqref{eq:cal-modal-monodromy}.
\section{Timing and observability in mutual adaptation}
\label{sec:consequences}

For continuous deployment, a simple negative branch in the delay-induced range
\(c_\ast<L=K|\gamma|<H-c_\ast\) depends on the four hard lags only through
\(\tau_\Sigma\). Hence the chain rule and the transversality calculation in
Appendix~\ref{app:sec4} give, at the first crossing,
\begin{equation}
\frac{\partial\operatorname{Re}z}{\partial\sigma_X}
=\frac{\partial\operatorname{Re}z}{\partial\sigma_Y}
=\frac{\partial\operatorname{Re}z}{\partial\theta_X}
=\frac{\partial\operatorname{Re}z}{\partial\theta_Y}
=:D_\gamma>0.
\label{eq:equal-hard-lag-leverage}
\end{equation}
Consequently, small positive reductions \(d_j\) of the four hard lags produce
\begin{equation}
\Delta\operatorname{Re}z
=-D_\gamma\sum_jd_j+o(\|d\|).
\label{eq:cost-weighted-root-shift}
\end{equation}
Thus equal amounts of time removed from observation or deployment have the same
local dynamical effect. The model does not by itself rank the channels:
the choice of stage depends on marginal cost, feasibility, reliability, and
the reduction that can actually be achieved. If \(L\ge H-c_\ast\), reducing
hard lags alone cannot restore strict linear stability; an implementation rate, exploration
damping, or the interaction strength must change.

The componentwise time-shift transformation for delay-coupled networks
\cite{lucken2013reduction,lucken2015classification}, specialized to the present
single feedback cycle, gives a stronger statement than equality of linear
thresholds. Any two nonnegative allocations of
\((\sigma_X,\sigma_Y,\theta_X,\theta_Y)\) with the same \(\tau_\Sigma\) map
complete solutions of the full nonlinear equations into one another by constant
time shifts of \(p,q,x,y\). Corresponding periodic solutions therefore have the
same period and the same amplitude in every individual portfolio component.
Relative phases and synchronous observables need not agree, however, and the
same unshifted initial histories need not produce corresponding transients or
basin readings. Period and component amplitudes alone therefore cannot identify
how this delay budget is divided between observation and implementation; that
requires additional phase or timing information. The explicit specialization is
recorded in Appendix~\ref{app:sec5}. For a periodic revision signal
\(p(t)=\sum_np_ne^{in\Omega t}\), the deployment filter gives
\begin{equation}
x_n=\frac{\kappa_Xe^{-in\Omega\theta_X}}
{\kappa_X+in\Omega}\,p_n,
\label{eq:deployment-Fourier-response}
\end{equation}
so the \(n\)-th harmonic is attenuated by
\(\kappa_X/\sqrt{\kappa_X^2+n^2\Omega^2}\) and shifted by
\(-n\Omega\theta_X-\arctan(n\Omega/\kappa_X)\) modulo \(2\pi\).

For continuous deployment, the discrete hard lags therefore form a cycle-delay
budget. Periodic deployment introduces additional timing variables, namely the
release period and the reset fraction. In the zero-hard-lag setting of
Section~\ref{sec:calendarized-fielding}, Theorem~\ref{thm:cal-critical-shift}
shows that, for sufficiently short release periods, changing this release
geometry changes the critical interaction gain even after the mean
eligibility-to-deployment waiting time has been matched under the uniform-phase
calibration. Reducing a hard lag and changing the release cadence are
consequently distinct interventions.

Aggregate performance need not resolve the underlying compositional adjustment.
For the scalar observable identities below, assume in addition the left
barycentric balance conditions
\begin{equation}
A^\top u=a\mathbf 1,\qquad B^\top u=b\mathbf 1.
\label{eq:observable-balance}
\end{equation}
The scalars necessarily agree with those in
\eqref{eq:equilibrium-balance}. Together, the right and left conditions are
called two-sided barycentric balance. The additional left conditions are used
only for the observable identities.

Let
\[
\Phi_X(x,y)=x^\top Ay,
\qquad
\Phi_Y(y,x)=y^\top Bx,
\]
and write \(\Phi_X^\ast=\Phi_X(u,u)\),
\(\Phi_Y^\ast=\Phi_Y(u,u)\).

\begin{proposition}[Balanced performance observables]
\label{prop:balanced-performance-observables}
Under \eqref{eq:equilibrium-balance} and \eqref{eq:observable-balance},
\begin{equation}
\Phi_X(u+\xi,u+\eta)-\Phi_X^\ast=\xi^\top A\eta,
\qquad
\Phi_Y(u+\eta,u+\xi)-\Phi_Y^\ast=\eta^\top B\xi.
\label{eq:balanced-performance-bilinear}
\end{equation}
If a simple Hopf mode has
\[
\xi(t)=\varepsilon\operatorname{Re}(ve^{i\omega_\ast t})
+\mathcal O(\varepsilon^2),
\qquad
\eta(t)=\varepsilon\operatorname{Re}(we^{i\omega_\ast t})
+\mathcal O(\varepsilon^2),
\]
then
\begin{equation}
\begin{aligned}
\Phi_X(t)-\Phi_X^\ast
&=\frac{\varepsilon^2}{2}\operatorname{Re}(v^\top A\overline w)\\
&\quad+\frac{\varepsilon^2}{2}
\operatorname{Re}\!\left(v^\top Aw\,e^{2i\omega_\ast t}\right)
+\mathcal O(\varepsilon^3),
\end{aligned}
\label{eq:section5-PhiX-second-harmonic}
\end{equation}
with the analogous formula for \(\Phi_Y\). Thus the leading scalar response is
generically a mean shift plus a second harmonic. Under strict antagonism,
\begin{equation}
\Phi_Y(y,x)-\Phi_Y^\ast
=-\chi\bigl[\Phi_X(x,y)-\Phi_X^\ast\bigr],
\label{eq:antiphase-locking-exact}
\end{equation}
so every nonzero Fourier harmonic of the two deviations has amplitude ratio
\(\chi\) and phase difference \(\pi\) modulo \(2\pi\).
\end{proposition}

The proof is given in Appendix~\ref{app:sec5}. The second harmonic follows from the two-sided balance conditions, not from
strict antagonism. If the residual linear derivative of the
observable is \(O(\varepsilon_b)\), the fundamental component is
\(O(\varepsilon\varepsilon_b)\), whereas the balanced second harmonic is
\(O(\varepsilon^2)\) when its quadratic coefficient is nonzero. Weak imbalance
therefore produces peaks near both \(\omega_\ast\) and \(2\omega_\ast\), and
noise may obscure either peak. For fixed \(\varepsilon_b>0\), the fundamental
component eventually dominates as \(\varepsilon\to0\); a pronounced second
harmonic therefore requires the distinguished regime
\(\varepsilon_b\ll\varepsilon\). The \(2{:}1\) relation is a diagnostic tendency
of an approximately balanced regime rather than a universal requirement for
observations. Section~\ref{sec:num} compares these consequences with the common
baseline example.
\section{Three-strategy illustration of the timing mechanisms}
\label{sec:num}

We use one dimensionless competition between portfolios of three existing
configurations to compare the continuous and periodic implementation mechanisms.
The abstract strategy labels isolate the interaction geometry, rather than
ranking particular technologies. The example checks local stability, nonlinear
portfolio motion, and its appearance in aggregate performance. Numerical
procedures, initial histories, and the lag-allocation and step-refinement checks
are given in Appendix~\ref{app:numerics}. Let
\begin{subequations}
\begin{align}
A&=\frac1{12}
\begin{pmatrix}
7&-5&-2\\
-5&7&-2\\
-2&-2&4
\end{pmatrix},
& B&=-A,
\label{eq:num-baseline-A}\\
\mu_X&=\mu_Y=\kappa_X=\kappa_Y=1,
& \lambda_X&=\lambda_Y=\frac92.
\label{eq:num-baseline-rates}
\end{align}
\end{subequations}
The two-sided barycentric balance conditions hold, and
\(\mathcal C=-\mathcal A\mathcal A^\top\) has eigenvalues
\(-1\) and \(-1/4\). Hence \(K=9/4\), the continuous branch strengths are
\(L_1=9/4\) and \(L_2=9/16\), and only \(L_1\) lies in the window
\(1<L<4\). Its frequency and first critical hard delay are
\begin{equation}
\omega_\ast=\frac1{\sqrt2},
\qquad
\tau_{\rm crit}
=\frac{\pi-4\arctan(1/\sqrt2)}{1/\sqrt2}
\approx0.9612.
\label{eq:num-baseline-taucrit}
\end{equation}

Figure~\ref{fig:num-stability-maps} compares the two linear timing mechanisms.
Panel (a) shows the principal continuous-deployment hard-delay threshold as the
antagonistic branch strength varies. Panel (b) uses unit rates and locates the
periodic boundary by scanning increasing gains and refining the first detected
crossing of \(\operatorname{spr}(J_{\varrho_T}e^{A_\Lambda T})=1\).
The curve extends to the full-reset endpoint \(T/d=2\), \(\varrho_T=1\),
where the monodromy matrix is evaluated directly. On the sampled unit-rate
curve the boundary decreases from the continuous limit \(\Lambda_0=4\) to
approximately \(3.163535\), a loss of \(20.912\%\). This finite-period
behaviour is numerical, not a monotonicity theorem for arbitrary rates or an
extension of Theorem~\ref{thm:cal-critical-shift}. The local approximation
\(4-T^2/3\) is shown only for \(T/d\le0.5\); numerically,
\((4-\Lambda_{\rm cal}(T))/T^2\to1/3\) as \(T\to0\).
The unequal-rate check in Table~\ref{tab:periodic-unequal-rate} retains the
distinction between synchronized and phase-averaged finite-period boundaries,
whose normalized shifts approach the same \(\mathfrak c\).

\paragraph{A reusable mean-matched benchmark}
For Figure~\ref{fig:matched-mean-contrast}, take \(d=1\), \(T/d=1\),
\(\varrho_T=2/3\), and \(\Lambda=3.85\). The numerically located periodic
boundary is \(3.688467\), below this gain, whereas the continuous boundary
is \(4\), above it. Thus the two mechanisms give opposite local stability
predictions without changing the interaction gain or calibrated mean waiting
time. Direct integrations from the same size-\(10^{-4}\) initial perturbation
on \(0\le t/d\le200\) give fitted exponents \(-0.009509803\) and
\(0.010915483\), in agreement with the eigenvalue and Floquet predictions
\(-0.009509796\) and \(0.010915576\). The final normalized amplitudes are
approximately \(0.1493\) and \(8.8737\). At all monitored solver nodes and
post-release samples, every portfolio component remains within
\(5.34\times10^{-4}\) of its equilibrium value, so the more visible
separation is still a local comparison.

The shorter-period matched pair is retained in the supplement:
\(T/d=0.5\), \(\varrho_T=0.4\), and \(\Lambda=3.95\), with the same
initial-perturbation convention and \(0\le t/d\le500\). Its fitted
exponents are \(-0.003139763\) and \(0.002067548\), compared with the
linear predictions \(-0.003139756\) and \(0.002067570\).
It checks the same discrepancy closer to the short-period regime rather than
replacing the leading-order convergence test.

Table~\ref{tab:matched-mean-benchmark} quantifies the change with cadence at
fixed mean lag. Its finite-period entries are obtained from monodromy rather
than by extrapolating the quadratic formula; the last row is the limiting
full-reset case. These are normalized benchmark values, not empirical
estimates for any application. The supplementary script
\texttt{implementation\_benchmark.py} reproduces the table, the expanded
curve, both matched nonlinear comparisons, and the periodic check in
Figure~\ref{fig:num-time-domain}(b), and accepts user-specified branch rates,
mean lag, period, and gain for the linear comparison.

\begin{table}[htbp]
\centering
\caption{Mean-matched implementation benchmark. The branch rates
 $\mu_X=\mu_Y=\kappa_Y=1$ and mean waiting time $d=1$ are fixed, all hard
 lags are zero, and the continuous gain boundary is
 $\Lambda_0=4$. The last column is
 $100(\Lambda_0-\Lambda_{\rm cal})/\Lambda_0$. Finite-period values are numerical;
 $T/d=2$ is the limiting full-reset case.}
\label{tab:matched-mean-benchmark}
\small
\begin{tabular}{@{}cccc@{}}
\toprule
$T/d$ & Release fraction $\varrho_T$ & Periodic boundary & Boundary loss (\%) \\
\midrule
0.25 & 0.222222 & 3.979222 & 0.519 \\
0.50 & 0.400000 & 3.917659 & 2.059 \\
1.00 & 0.666667 & 3.688467 & 7.788 \\
1.50 & 0.857143 & 3.391793 & 15.205 \\
2.00 & 1.000000 & 3.163535 & 20.912 \\
\bottomrule
\end{tabular}
\end{table}

\begin{figure}[t]
\centering
\begin{minipage}[t]{0.485\textwidth}
\centering
\includegraphics[width=\linewidth]{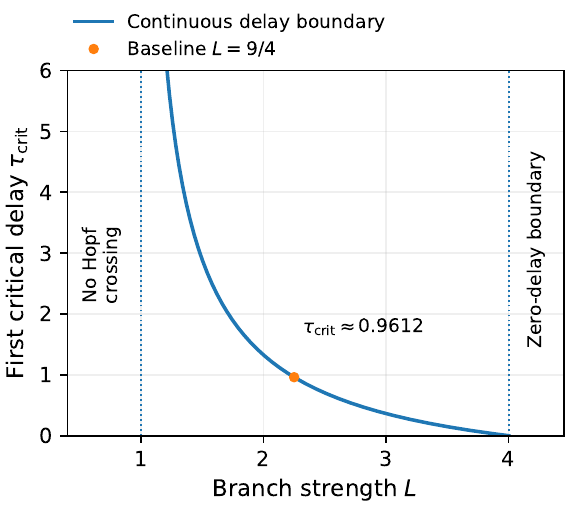}\\[-0.3em]
{\small (a)}
\end{minipage}\hfill
\begin{minipage}[t]{0.485\textwidth}
\centering
\includegraphics[width=\linewidth]{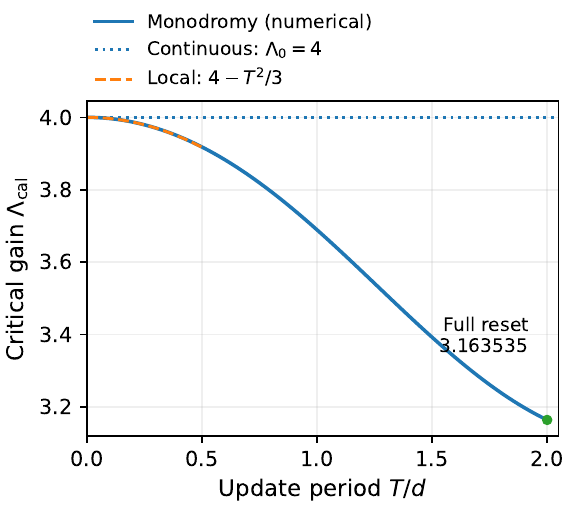}\\[-0.3em]
{\small (b)}
\end{minipage}
\caption{Timing mechanisms and local stability boundaries. (a) Principal critical
hard delay versus antagonistic branch strength for unit rates; the marked
baseline is \(L=9/4\). (b) Unit-rate mean-matched periodic deployment:
numerically computed monodromy boundary through the full-reset endpoint
\(T/d=2\), local approximation \(4-T^2/3\) shown only for \(T/d\le0.5\),
and continuous boundary \(\Lambda_0=4\). The calibrated mean \(d=1\) is
fixed; the finite-period curve is a numerical observation.}
\label{fig:num-stability-maps}
\label{fig:num-gain-window}
\label{fig:num-calendarized-gain}
\end{figure}

\FloatBarrier

\begin{figure}[t]
\centering
\begin{minipage}[t]{0.485\textwidth}
\centering
\includegraphics[width=\linewidth]{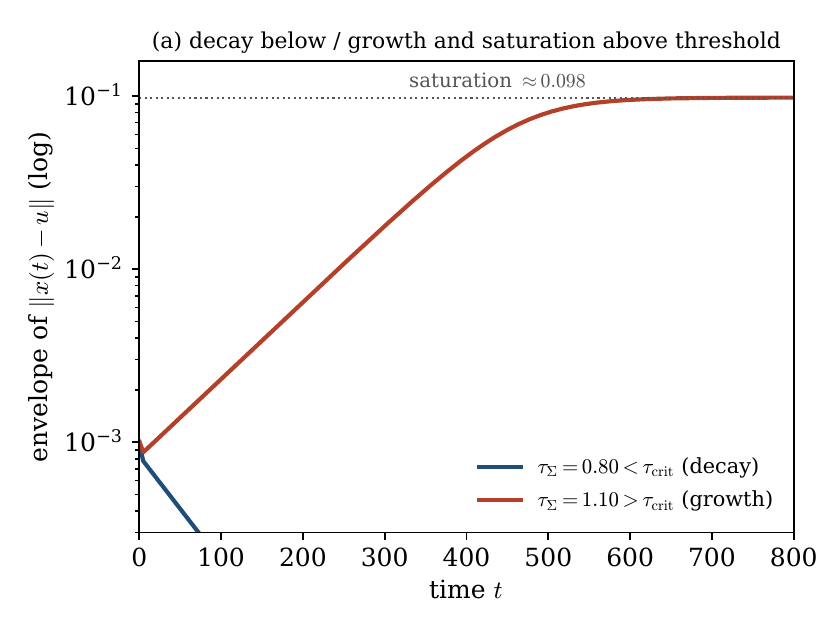}
\end{minipage}\hfill
\begin{minipage}[t]{0.485\textwidth}
\centering
\includegraphics[width=\linewidth]{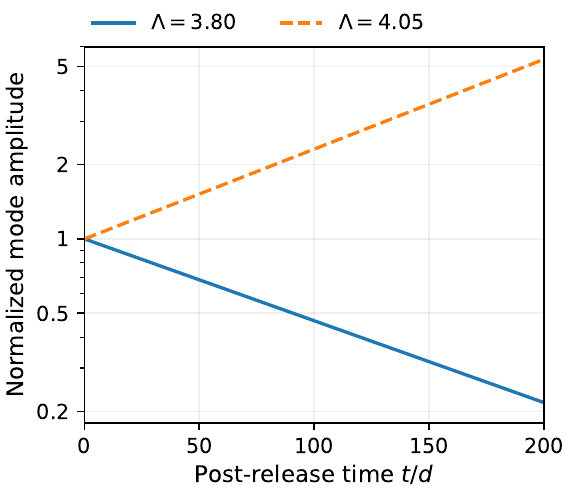}
\end{minipage}
\caption{Growth and decay of portfolio perturbations in the full nonlinear
system. (a) Continuous deployment below and
above its critical total hard delay. (b) Periodic partial deployment at
\(T=0.5\), \(\varrho_T=0.4\), and zero hard lags: normalized stroboscopic
amplitude of the leading singular mode for \(\Lambda=3.80\) and \(4.05\),
around \(\Lambda_{\rm cal}=3.91766\). The fitted/Floquet exponents are
\(-0.0076319/-0.0076317\) and \(0.0083761/0.0083812\). The vertical axis is logarithmic.}
\label{fig:num-time-domain}
\end{figure}

\FloatBarrier

\begin{figure}[!t]
\centering
\includegraphics[width=0.62\textwidth]{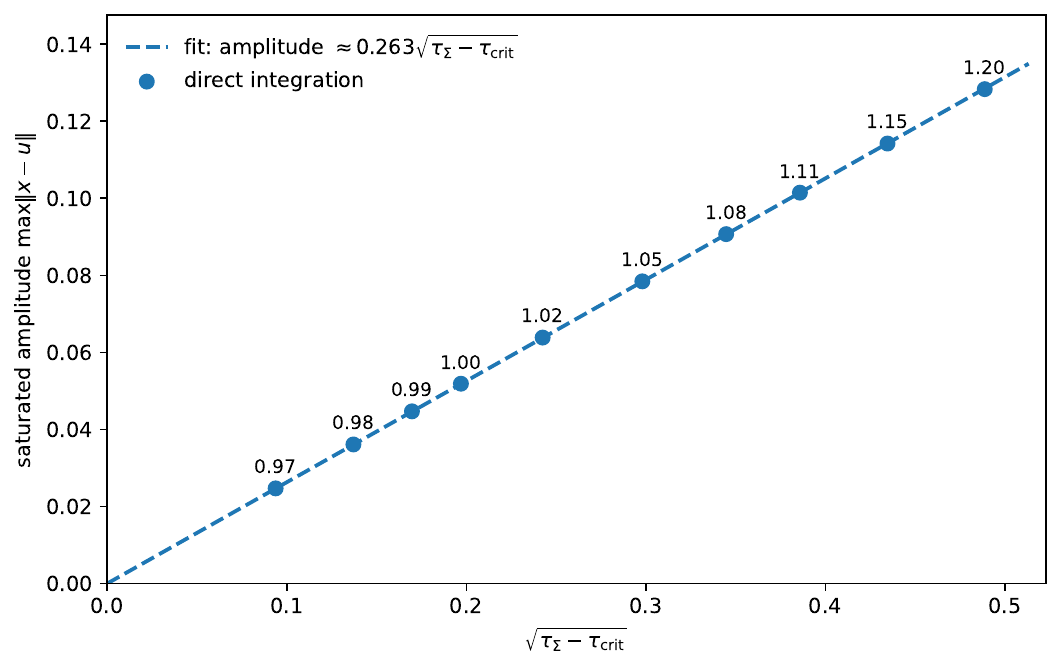}
\par\smallskip
\includegraphics[width=0.72\textwidth]{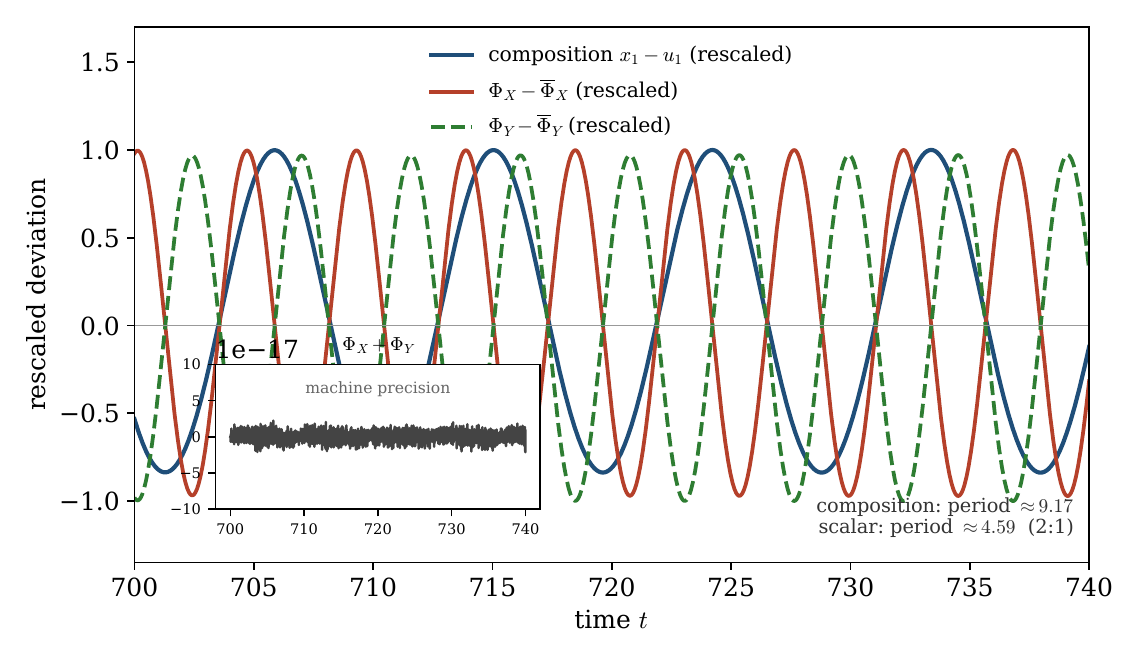}
\caption{Portfolio motion and its aggregate performance signature. Top: saturated amplitude versus
\(\sqrt{\tau_\Sigma-\tau_{\rm crit}}\); the fitted slope is \(0.263\), while
the normal form predicts \(0.26345\). Bottom: the deployed composition has
period approximately \(9.17\), while the balanced scalar performances have
period approximately \(4.59\) and are locked in antiphase. Curves in the lower
panel are centered and rescaled.}
\label{fig:num-nonlinear-signatures}
\label{fig:num-amplitude-scaling}
\label{fig:num-observable}
\end{figure}

For the equal lag split, Appendix~\ref{app:baseline-lyapunov} gives
\(\operatorname{Re}c_H=-0.9193272749<0\), so the local continuous-deployment
Hopf branch is supercritical and locally orbitally attracting, with predicted
deployed-amplitude coefficient \(0.26345\). Figure~\ref{fig:num-time-domain}(a)
shows decay at \(\tau_\Sigma=0.80\) and growth followed by bounded oscillation
at \(\tau_\Sigma=1.10\); panel (b) confirms local decay and growth on the two
sides of the periodic Floquet boundary. The second continuous spectral branch
has no positive-frequency crossing. Bounded saturation in panel (a) is a
numerical observation; neither global attraction nor a postcritical hybrid
attractor is claimed.

To check the fixed-total-delay consequence of Section~\ref{sec:consequences},
the same \(\tau_\Sigma=1.10\) was redistributed among equal,
\(X\)-cross-heavy, and \(Y\)-cross-heavy allocations. Table~\ref{tab:delay-split}
in Appendix~\ref{app:numerics} shows agreement of the componentwise periods and
amplitudes and the predicted change of deployment phase; because all runs use
the same unshifted history, no equality of transients or basin readings is
asserted.

Figure~\ref{fig:num-nonlinear-signatures} collects the remaining nonlinear
checks. The fitted amplitude slope is \(0.263\), close to the normal-form value
\(0.26345\). At \(\tau_\Sigma=1.10\), the measured compositional and scalar
periods are \(9.170\) and \(4.585\), while the first-order normal-form estimates
are approximately \(9.18\) and \(4.59\). The lower panel also shows the exact
baseline relation
\[
\Phi_Y(t)-\Phi_Y^\ast
=-\bigl[\Phi_X(t)-\Phi_X^\ast\bigr].
\]
Thus two-sided balance produces the second harmonic, whereas strict antagonism
fixes the
cross-population phase and amplitude relation. The checks with \(\chi=2,3\)
are reported in Appendix~\ref{app:numerics}.

\FloatBarrier
\section{Discussion}
\label{sec:disc}

The motivating problem is the gap between revising a technological portfolio
and changing the portfolio actually in use while another population is also
adapting. The model separates revision responsiveness, observation lag, gradual
implementation, and scheduled partial updates. Its principal comparison shows
that a mean implementation time is not, by itself, a sufficient description of
the temporal feedback: under the assumptions of the periodic theorem, matching
that mean leaves a downward quadratic shift of the first gain boundary. The
three-strategy example makes the consequence explicit through opposite
continuous and periodic stability verdicts at the same calibrated mean lag.

The spectral reduction and observable identities answer complementary questions.
Under strict antagonism, the first unstable mode is a paired redistribution of
the two portfolios identified by singular vectors, not a ranking of individual
technologies. With continuous implementation, the hard lags form one cycle-delay
budget; period and component amplitudes do not identify its allocation.
Changing release cadence is a different operation from shortening a hard lag.
Under two-sided balance, aggregate performance responds only at quadratic order
to leading compositional motion and can show a different apparent period. Thus
neither a mean timing statistic nor a single performance trace necessarily
resolves the dynamics of mutual adaptation.

\subsection{Application questions and reuse of the benchmark}
\label{sec:applications}

The benchmark supports a controlled comparison across several settings: keep
the adaptation mechanism and mean implementation lag fixed, and compare gradual
implementation with scheduled activation. The following applications give
concrete state variables and experiments for that comparison.

\paragraph{Adaptive cyber defence}
An intended mix of security configurations can differ from the configurations
currently active while attack strategies also adapt. Evolutionary models of
this interaction provide a starting point
\cite{bashir2026cyber,tan2025honeypoint}. The proposed experiment holds the
interaction, revision rates, and mean implementation lag fixed and compares
gradual configuration rollout with periodic partial updates. The outputs are
the critical gain and the local growth rate of configuration perturbations:
quantities that distinguish an implementation effect from a change in the
underlying adaptive interaction.

\paragraph{Algorithmic trading}
Coevolution among adaptive automated traders has been studied in market
simulations \cite{cliff2022trading}. Here the corresponding implementation
question concerns the intended allocation across trading strategies and the
mix actually activated after validation or scheduled model releases. One can
compare continuous and periodic activation in an agent-based market while
holding the learning rule and mean activation lag fixed, and measure the
growth of perturbations in the active strategy mix. This distinguishes
strategy revision from strategy use; the relevant portfolio is an allocation
across algorithms, not necessarily an allocation across financial assets.
A market model is needed to connect compositional dynamics to prices or profit.

\paragraph{Software release scheduling}
Partial rollout and the separation of feature activation from binary releases
are established release-engineering mechanisms \cite{warner2018canary}. For
competing configurable services, the proposed comparison concerns intended
versus active configuration mixes and asks whether a continuous rollout
approximation misses an instability introduced by scheduled activation. This makes the cadence of feature activation a separately testable element
of a competitive feedback model.

\paragraph{Cybergenetic control}
Computer-mediated feedback has been used to regulate bacterial co-culture
composition \cite{gutierrez2022dynamic}. This suggests a related extension in
which commanded actuation and the realized biological response are separated,
and continuous and periodic feedback updates are compared at matched mean lag.
The external controller is not automatically a second replicator population;
its feedback law must be incorporated before using the quantitative boundary.

\paragraph{How to use the comparison}
Within the two-population setting, first specify the strategy repertoire,
intended and active shares, interaction matrices, and revision and
implementation rates. The balanced equilibrium and tangent antagonism are
conditions to check, not properties obtained by simply relabelling data. For
the zero-hard-lag periodic comparison, record the period and release fraction
and verify their mean-lag relation in \eqref{eq:cal-mean-matched-reset}.
Then compute the singular-mode gains, the continuous boundary, and the
monodromy multipliers using \eqref{eq:cal-modal-gain},
\eqref{eq:cal-continuous-boundary}, and \eqref{eq:cal-modal-monodromy}.
The supplementary benchmark implements this last step. A disagreement asks
whether a mean-only timing approximation changes the stability conclusion;
uncertainty intervals overlapping a boundary require a parameter-range
comparison rather than a sharp verdict. Outside the structural assumptions,
the same controlled comparison remains a research design, but the relevant
stability problem must be solved for the application model.

\subsection{Scope and next empirical step}

The analysis is local to a uniform interior reference equilibrium and uses the
specified uniform exploration law on a fixed strategy repertoire. Stability
here means decay of small compositional perturbations, not technological
superiority or success in a competition. The model does not describe the
invention of new configurations, absolute capacities, or growth in those
capacities. The periodic theorem treats one common release calendar without
hard discrete lags. Repeated critical singular values, mode-dependent resets,
and two-sided periodic implementation require additional analysis. Within these
limits, the results make the distinction between deciding and implementing a
quantitative part of the adaptation model, rather than an unspecified delay.
The next empirical step is to measure revision and activation separately in
one of these settings and test whether the observed difference between
implementation schedules persists in its own interaction model. The present
benchmark supplies a reproducible reference for that test.

\paragraph{Computational availability}
The supplementary benchmark includes a documented Python script, parameter
settings, machine-readable boundary and trajectory data, and numerical
self-checks. It reproduces Table~\ref{tab:matched-mean-benchmark},
Figure~\ref{fig:matched-mean-contrast}, both boundary panels in
Figure~\ref{fig:num-stability-maps}, the nonlinear periodic panel in
Figure~\ref{fig:num-time-domain}(b), and a separate shorter-period matched
comparison. It is not a complete reproduction package for the earlier
delay-equation and normal-form calculations, whose procedures are described in
Appendix~\ref{app:numerics}.

\FloatBarrier
\appendix
\renewcommand{\thesection}{\Alph{section}}
\renewcommand{\thetheorem}{\thesection.\arabic{theorem}}
\setcounter{table}{0}

\section{Spectral reduction and continuous-branch details}
\label{app:spectral-continuous}
\label{app:technical}
\label{app:sec2}
\label{app:sec4}

\begin{proof}[Forward invariance and deployment kernels]
We give the argument for \(p\) and \(x\); the equations for \(q\) and \(y\)
are identical. With \(s_p=\mathbf 1^\top p\), summing
\eqref{eq:p-dynamics} gives
\[
\dot s_p
=\bigl[\lambda_Xp^\top Ay(t-\sigma_X)+\mu_X\bigr](1-s_p).
\]
Hence \(s_p(0)=1\) implies \(s_p(t)=1\). At a boundary face \(p_k=0\),
the replicator term vanishes and \(\dot p_k=\mu_X/m>0\). Moreover,
\[
x(t)=e^{-\kappa_Xt}x(0)
+\kappa_X\int_0^t e^{-\kappa_X(t-s)}p(s-\theta_X)\,ds.
\]
The coefficients form a nonnegative unit-mass combination, so
\(x(t)\in\Delta_m\) whenever the history lies in \(\Delta_m\). Thus
\(\Delta_m^4\) is forward invariant.

For a complete bounded trajectory, letting the initial time tend to
\(-\infty\) gives
\[
x(t)=\int_{\theta_X}^{\infty}
\kappa_Xe^{-\kappa_X(a-\theta_X)}p(t-a)\,da,
\]
and similarly for \(y\). For an initial-value solution, the difference from
this convolution is an exponentially decaying homogeneous term. This is the
one-stage exponential distributed-delay representation of a constant-rate
compartment; see \cite{cassidy2021distributed} for the general cyclic setting.
\end{proof}

\begin{proof}[Proof of Proposition~\ref{prop:characteristic-reduction}]
Under \eqref{eq:equilibrium-balance}, the first variation of the \(X\)-revision
field in tangent directions \((\alpha,\eta)\) is
\[
-\mu_X\alpha+\frac{\lambda_X}{m}PAP\,\eta,
\]
and the analogous formula holds for \(Y\). Thus
\eqref{eq:revision-linear} requires no left balance condition. With the modal
ansatz
\[
(\alpha,\beta,\xi,\eta)=(a,b,r,s)e^{zt},
\]
equations \eqref{eq:revision-linear}--\eqref{eq:deployment-linear} become
\[
\begin{aligned}
(z+\mu_X)a&=\frac{\lambda_X}{m}\mathcal A s e^{-z\sigma_X},&
(z+\mu_Y)b&=\frac{\lambda_Y}{m}\mathcal B r e^{-z\sigma_Y},\\
(z+\kappa_X)r&=\kappa_Xa e^{-z\theta_X},&
(z+\kappa_Y)s&=\kappa_Yb e^{-z\theta_Y}.
\end{aligned}
\]
Eliminating \(b,r,s\), without dividing by a filter factor, yields
\[
R(z)a=Ke^{-z\tau_\Sigma}\mathcal A\mathcal B a.
\]
To justify the full characteristic determinant, including the pole values,
let \(\Delta_{\rm full}(z)\) denote the full characteristic matrix in tangent
coordinates. For
\(z\notin\{-\mu_X,-\mu_Y,-\kappa_X,-\kappa_Y\}\), block elimination gives
\[
 \det\Delta_{\rm full}(z)
 =\det\!\left[R(z)I-Ke^{-z\tau_\Sigma}\mathcal C\right].
\]
Both determinants are entire functions of \(z\), so this identity extends to
the excluded points by analytic continuation. This proves
\eqref{eq:uniform-determinant-factor}, including the characteristic roots and
their algebraic multiplicities at those points. Since
\[
\det\!\left[R(z)I-Ke^{-z\tau_\Sigma}\mathcal C\right]
=\prod_{j=1}^{m-1}
\left[R(z)-Ke^{-z\tau_\Sigma}\gamma_j\right],
\]
with eigenvalues counted algebraically, diagonalizability of \(\mathcal C\) is
not required, and every \(\gamma_j\) gives
\eqref{eq:deployment-characteristic}.
\end{proof}

\begin{proof}[Continuous-branch classification and Theorem~\ref{thm:delay-induced-hopf}]
At zero hard delay a negative branch is
\[
R(z)+L=z^4+c_3z^3+c_2z^2+c_1z+(c_\ast+L).
\]
The quartic Routh--Hurwitz criterion reduces to
\[
(c_3c_2-c_1)c_1>c_3^2(c_\ast+L),
\]
because all coefficients are positive and
\[
c_3c_2-c_1
=2c_1+\sum_{1\le i<j\le4}r_ir_j(r_i+r_j)>0.
\]
Thus zero-delay stability is equivalent to \(L<H-c_\ast\). After normalization
by \(c_\ast\), this is the standard four-stage negative-feedback cycle. The
classical secant criterion and its delayed extension
\cite{arcak2006secant,ahsen2015secant} give
\[
H-c_\ast\ge4c_\ast,
\]
with equality for equal rates; the exact boundary \(H-c_\ast\) above is used
in the remainder of the argument.

For an imaginary root,
\[
L=|R(i\omega)|
=\left[\prod_{j=1}^4(r_j^2+\omega^2)\right]^{1/2}.
\]
The right-hand side increases strictly from \(c_\ast\), so a unique positive
frequency exists exactly when \(L>c_\ast\). If also \(L<H-c_\ast\), its phase
satisfies \(\phi(\omega_\ast)<\pi\), and the first crossing is
\eqref{eq:first-delay-branch}. At such a root, implicit differentiation of
\(F(z,\tau)=R(z)+Le^{-z\tau}\) gives
\[
\frac{dz}{d\tau}=-\frac{z}{R'(z)/R(z)+\tau}.
\]
Here
\[
a(\omega)=\sum_{j=1}^4\frac{r_j}{r_j^2+\omega^2},
\qquad
b(\omega)=\sum_{j=1}^4\frac{\omega}{r_j^2+\omega^2}.
\]
Using \(R'(i\omega)/R(i\omega)=a(\omega)-ib(\omega)\) yields
\begin{equation}
\left.\frac{d}{d\tau}\operatorname{Re}z(\tau)
\right|_{\tau=\tau_{\rm first}}
=
\frac{\omega_\ast b(\omega_\ast)}
{\bigl(a(\omega_\ast)+\tau_{\rm first}\bigr)^2+b(\omega_\ast)^2}>0.
\label{eq:transversality-formula}
\end{equation}
Thus every imaginary root is simple and crosses from left to right. Moreover,
\[
\frac{d\omega_\ast}{dL}=\frac{1}{Lb(\omega_\ast)},
\qquad
\frac{d\tau_{\rm first}}{dL}
=-\frac{\omega_\ast a(\omega_\ast)+\pi-\phi(\omega_\ast)}
{\omega_\ast^2Lb(\omega_\ast)}<0.
\]

Since \(R(0)+L>0\), zero is never a root. Standard continuity of characteristic
roots gives the three cases in \eqref{eq:delay-induced-window}: branches with
\(L\le c_\ast\) have no imaginary crossing, branches with
\(c_\ast<L<H-c_\ast\) first cross at \(\tau_{\rm first}(L)\), and branches
with \(L\ge H-c_\ast\) are marginal or unstable already at zero hard delay.
The left-to-right orientation shows that nonnegative hard delays cannot restore
strict linear stability in the last case. At equality and zero hard delay,
the linearization is marginal; nonlinear stability there requires a separate
nonlinear analysis.

Under strict antagonism, \(\mathcal C=-\chi\mathcal A\mathcal A^\top\) is
symmetric nonpositive, while its zero eigenvalues contribute only the stable
roots of \(R(z)=0\). The assumptions of Theorem~\ref{thm:delay-induced-hopf}
therefore leave one simple conjugate pair at the unique minimizing branch and
no earlier crossing. The crossing derivative is
\eqref{eq:transversality-formula} for every \(C^1\) lag family parameterized by
its sum. The standard Hopf theorem for retarded functional differential
equations \cite{hale1993,hassard1981} then gives the stated local periodic
branch whenever the first Lyapunov coefficient is nonzero.
\end{proof}
\section{Mean-matched periodic deployment: proof details}
\label{app:periodic-deployment}
\label{app:calendarized-fielding}

The one-period map and Floquet criterion are standard for periodic impulsive
systems; see, for example, \cite{churilov2016goodwin}. We record only the
modal reduction and the perturbation calculation specific to the present model.

\begin{proof}[Proof of Proposition~\ref{prop:cal-modal-reduction}]
Write
\(\alpha=a u_j\), \(\xi=xu_j\), \(\beta=bv_j\), and \(\eta=yv_j\).
Using \(\mathcal B=-\chi\mathcal A^\top\) in
\eqref{eq:cal-linear-flow} gives
\[
 \dot a=-\mu_Xa+\frac{\lambda_X}{m}s_jy,
 \quad
 \dot b=-\mu_Yb-\frac{\chi\lambda_Y}{m}s_jx,
 \quad
 \dot x=0,
 \quad
 \dot y=\kappa_Y(b-y),
\]
and the reset is \(x^+=(1-\varrho)x^-+\varrho a^-\). A constant rescaling of
\((a,x)\) relative to \((b,y)\) makes the two cross-couplings
\(\pm\sqrt{\Lambda_j}\), proving
\eqref{eq:cal-modal-monodromy}--\eqref{eq:cal-branch-matrices}. Since
\(\mathcal A\mathcal A^\top\) is symmetric, its spectral eigenspaces form an
orthogonal decomposition and the common scalar reset acts identically on each
of them. A general tangent reset operator \(\mathcal R\) preserves all these
eigenspaces exactly when it commutes with
\(\mathcal A\mathcal A^\top\), which proves
\eqref{eq:cal-commuting-reset}. Within a repeated eigenspace, this condition
allows mixing of the chosen singular-vector basis.
\end{proof}

\begin{lemma}[Persistence of the first gain boundary]
\label{lem:cal-first-gain-boundary}
Let \(G_T(\Lambda)\) be a real matrix family that is continuous on
\([0,T_1]\times[0,\Lambda_0+\delta_0]\) and continuously differentiable in
\(\Lambda\) near \((0,\Lambda_0)\). Suppose that \(G_T\to G_0\) uniformly on
the compact gain interval and that \(\partial_\Lambda G_T\to
\partial_\Lambda G_0\) uniformly near \(\Lambda_0\). Assume that
\(G_0(\Lambda)\) is Hurwitz for \(0\le\Lambda<\Lambda_0\), while
\(G_0(\Lambda_0)\) has exactly one simple pair \(\pm i\omega_0\) on the
imaginary axis and all other eigenvalues are strictly stable. If the critical
eigenvalue branch satisfies
\(
 \partial_\Lambda\operatorname{Re}z_0(\Lambda_0)>0,
\)
then, for all sufficiently small \(T>0\), there is a unique nearby crossing
\(\Lambda_T\to\Lambda_0\), and \(\Lambda_T\) is the first gain at which
\(G_T(\Lambda)\) loses Hurwitz stability.
\end{lemma}

\begin{proof}
Choose \(\eta>0\) so that every noncritical eigenvalue of
\(G_0(\Lambda_0)\) has real part less than \(-4\eta\). By simplicity and
transversality, choose \(0<\delta<\min\{\delta_0/2,\Lambda_0/2\}\) so that,
for \(|\Lambda-\Lambda_0|\le2\delta\), the critical conjugate pair remains
simple, its real part has positive \(\Lambda\)-derivative and absolute value
less than \(\eta\), and every noncritical eigenvalue has real part less than
\(-3\eta\). Independently, compactness and the Hurwitz property give
\(\varepsilon>0\) such that the spectral abscissa of \(G_0(\Lambda)\) is at
most \(-3\varepsilon\) on \([0,\Lambda_0-\delta]\).

Uniform convergence and finite-dimensional spectral continuity preserve this
compact-interval stability gap for all sufficiently small \(T\). In the local
interval, simplicity and the assumed \(C^1\)-convergence continue the critical
pair with a strictly increasing real part, while all noncritical eigenvalues
remain in \(\{\operatorname{Re}z<-2\eta\}\). At
\(\Lambda_0-\delta\) and \(\Lambda_0+\delta\), the continued critical real
part retains, respectively, its negative and positive signs. Continuity and
strict monotonicity therefore give exactly one crossing \(\Lambda_T\) between
these endpoints. Uniform convergence of the critical branch and its unique
zero imply \(\Lambda_T\to\Lambda_0\). Combining the compact interval with
the local interval shows that \(G_T(\Lambda)\) is Hurwitz for every
\(0\le\Lambda<\Lambda_T\), has only the critical pair on the imaginary
axis at \(\Lambda_T\), and is unstable immediately above it.
\end{proof}

\begin{proof}[Proof of Theorem~\ref{thm:cal-critical-shift}]
Let
\[
 N=
 \begin{pmatrix}
 0&0&0&0\\
 0&0&0&0\\
 1&0&-1&0\\
 0&0&0&0
 \end{pmatrix},
 \qquad N^2=-N.
\]
Then \(J_{\varrho_T}=e^{\vartheta(T)N}\), where
\begin{equation}
 \vartheta(T)=-\log(1-\varrho_T)
 =\log\frac{d+T/2}{d-T/2}
 =\frac{T}{d}+\frac{T^3}{12d^3}+O(T^5).
\label{eq:cal-vartheta-expansion}
\end{equation}
The monodromy matrix is similar to the symmetric composition
\[
 S_\Lambda(T)=e^{\vartheta(T)N/2}e^{A_\Lambda T}
              e^{\vartheta(T)N/2}.
\]
Uniformly for \(\Lambda\) in any fixed compact interval,
\(S_\Lambda(T)=I+O(T)\). Hence its principal logarithm is defined for all
sufficiently small \(T\), and the modified generator
\(G_T(\Lambda)=T^{-1}\log S_\Lambda(T)\) converges uniformly to
\(G_0(\Lambda)=A_\Lambda+d^{-1}N\). Moreover,
\(S_\Lambda(-T)=S_\Lambda(T)^{-1}\), so the same logarithm branch gives
\(G_{-T}(\Lambda)=G_T(\Lambda)\). The standard symmetric BCH expansion
\cite{strang1968splitting,hairer2006geometric} therefore yields
\begin{equation}
 G_T(\Lambda)=G_0(\Lambda)+T^2G_2(\Lambda)+O(T^4),
\label{eq:cal-modified-generator}
\end{equation}
where \(G_0=A_\Lambda+d^{-1}N\) and
\begin{equation}
 G_2=
 \frac{1}{12d^3}N
 -\frac{1}{24d^2}[N,[N,A_\Lambda]]
 +\frac{1}{12d}[A_\Lambda,[A_\Lambda,N]].
\label{eq:cal-G2}
\end{equation}
We use this standard expansion only as an input. Direct evaluation of the
four-dimensional determinant gives the model-specific identity
\begin{equation}
\begin{aligned}
 d\det(zI-G_T(\Lambda))
 ={}&F(z,\Lambda)\\
 &+T^2\left[
 \frac{F(z,\Lambda)}{12d^2(1+dz)}
 -\frac{\kappa_Y\Lambda z^2}{12(1+dz)}
 \right]+O(T^4),
\end{aligned}
\label{eq:cal-determinant-expansion}
\end{equation}
locally uniformly near \((i\omega_0,\Lambda_0)\). At the continuous boundary,
\(F(i\omega_0,\Lambda_0)=0\), so the order-\(T^2\) correction is
\begin{equation}
 \kappa_Y\Lambda_0q,
 \qquad
 q=\frac{\omega_0^2}{12(1+id\omega_0)}.
\label{eq:cal-q}
\end{equation}

Put \(\Psi=D'(i\omega_0)/D(i\omega_0)\). From
\eqref{eq:cal-continuous-branch} and \eqref{eq:cal-a0-b0},
\begin{equation}
 \Psi=a_0-i\omega_0b_0,
 \qquad a_0,b_0>0.
\label{eq:cal-Psi}
\end{equation}
Hence \(D'(i\omega_0)\ne0\). The Jacobian of the real map
\[
 (\omega,\Lambda)\longmapsto
 \bigl(\operatorname{Re}F(i\omega,\Lambda),
       \operatorname{Im}F(i\omega,\Lambda)\bigr)
\]
has determinant
\[
 \kappa_Y^2\Lambda_0a_0>0.
\]
The implicit-function theorem therefore yields even analytic continuations
\(\omega_{\rm cal}(T)\) and \(\Lambda_{\rm cal}(T)\).
Writing
\[
 \omega_{\rm cal}=\omega_0+\nu_{\rm cal}T^2+O(T^4),
 \qquad
 \Lambda_{\rm cal}=\Lambda_0+\ell T^2+O(T^4),
\]
the order-\(T^2\) part of \eqref{eq:cal-determinant-expansion} is
\[
 i\nu_{\rm cal} D'(i\omega_0)+\kappa_Y\ell+\kappa_Y\Lambda_0q=0.
\]
Since \(D(i\omega_0)=-\kappa_Y\Lambda_0\) and \(D'=D\Psi\), separation of
real and imaginary parts gives
\[
 \nu_{\rm cal}=\frac{\operatorname{Im}q}{a_0},
 \qquad
 \ell=-\Lambda_0\left(
 \operatorname{Re}q+
 \operatorname{Im}q\frac{\operatorname{Im}\Psi}{a_0}\right).
\]
Using
\[
 \operatorname{Re}q=\frac{\omega_0^2}{12(1+d^2\omega_0^2)},
 \quad
 \operatorname{Im}q=-\frac{d\omega_0^3}{12(1+d^2\omega_0^2)},
 \quad
 \operatorname{Im}\Psi=-\omega_0b_0
\]
proves \eqref{eq:cal-coefficients} with \(\mathfrak c=-\ell>0\).
Finally,
\[
 \left.\frac{\partial z}{\partial\Lambda}\right|_{T=0}
 =-\frac{\kappa_Y}{D'(z)}
 =\frac{1}{\Lambda_0\Psi},
 \qquad
 \operatorname{Re}\frac{1}{\Lambda_0\Psi}
 =\frac{a_0}{\Lambda_0(a_0^2+\omega_0^2b_0^2)}>0.
\]
The implicit-function continuation therefore has positive crossing speed for
all sufficiently small \(T\). Since the monodromy matrix is similar to
\(S_\Lambda(T)=e^{TG_T(\Lambda)}\), its multipliers lie in the open unit disk
exactly when \(G_T(\Lambda)\) is Hurwitz. The continuous stability statement
following \eqref{eq:cal-continuous-boundary-RH}, the uniform convergence above,
and Lemma~\ref{lem:cal-first-gain-boundary} now show that
\(\Lambda_{\rm cal}(T)\) is the first positive periodic gain boundary, with no
other unit-modulus multipliers there.
\end{proof}

For the pole-matched calibration in
\eqref{eq:cal-pole-matched-shift},
\(J_{\varrho_T^{\rm pole}}=e^{(T/d)N}\). Using the same cited symmetric
composition expansion and evaluating only the model-specific determinant gives
\[
 d\det\!\left(zI-G_T^{\rm pole}(\Lambda)\right)
 =
 F(z,\Lambda)-\frac{\kappa_Y\Lambda z}{12d}\,T^2+O(T^4).
\]
At \(z=i\omega_0\), the same real--imaginary split gives
\eqref{eq:cal-pole-matched-shift} and
\[
 \omega_{\rm pole}(T)
 =\omega_0-\frac{\omega_0}{12d\,a_0}T^2+O(T^4).
\]
Lemma~\ref{lem:cal-first-gain-boundary} applies unchanged because
\(G_T^{\rm pole}\to G_0\) uniformly.

For the phase-averaged kernel, expansion of
\eqref{eq:cal-kernel-expansion} and multiplication of its branch equation by
\(1+dz\) give
\[
 F(z,\Lambda)-\frac{\kappa_Y\Lambda z^2}{12(1+dz)}T^2+O(T^4)=0.
\]
Its order-\(T^2\) term agrees with that of
\eqref{eq:cal-determinant-expansion} on the continuous characteristic locus
\(F(z,\Lambda)=0\); off that locus the determinant expansion contains the
additional term \(F(z,\Lambda)/[12d^2(1+dz)]\). Evaluating at the simple
continuous critical point therefore gives the same \(\nu_{\rm cal}\) and
\(\mathfrak c\). General dependence on the full delay distribution is
discussed in \cite{atay2008delayed}.

\section{Cycle-delay and observable details}
\label{app:cycle-observables}
\label{app:consequences}
\label{app:sec5}

The general componentwise time-shift transformation and invariance of cycle
delay sums are established in
\cite{lucken2013reduction,lucken2015classification}. For the present four-edge
cycle, let tildes denote a second lag allocation and put
\(\Delta\sigma_i=\widetilde\sigma_i-\sigma_i\) and
\(\Delta\theta_i=\widetilde\theta_i-\theta_i\). When the total delays agree,
one admissible set of component shifts is
\[
 a_p=0,\qquad
 a_y=\Delta\sigma_X,\qquad
 a_q=\Delta\sigma_X+\Delta\theta_Y,\qquad
 a_x=\Delta\sigma_X+\Delta\theta_Y+\Delta\sigma_Y.
\]
Then
\[
\widetilde p(t)=p(t+a_p),\quad
\widetilde q(t)=q(t+a_q),\quad
\widetilde x(t)=x(t+a_x),\quad
\widetilde y(t)=y(t+a_y)
\]
satisfies the equations with the second allocation. The shifts are unique up
to a common change of time origin. This specialization gives the period,
component-amplitude, phase, and history qualifications stated in
Section~\ref{sec:consequences}. Substitution of Fourier series into
\eqref{eq:x-deployment} gives
\[
(\kappa_X+in\Omega)x_n
=\kappa_Xe^{-in\Omega\theta_X}p_n,
\]
which is \eqref{eq:deployment-Fourier-response}.

\begin{proof}[Proof of Proposition~\ref{prop:balanced-performance-observables}]
Equations \eqref{eq:equilibrium-balance} and \eqref{eq:observable-balance},
together with \(\xi,\eta\in T\Delta_m\), give
\[
(u+\xi)^\top A(u+\eta)-u^\top Au=\xi^\top A\eta,
\]
and the analogous identity for \(B\), proving
\eqref{eq:balanced-performance-bilinear}. Multiplication of the two leading
real harmonics yields
\[
\xi(t)^\top A\eta(t)
=\frac{\varepsilon^2}{2}\operatorname{Re}(v^\top A\overline w)
+\frac{\varepsilon^2}{2}
\operatorname{Re}\!\left(v^\top Aw\,e^{2i\omega_\ast t}\right)
+\mathcal O(\varepsilon^3),
\]
which proves \eqref{eq:section5-PhiX-second-harmonic}; the calculation for
\(\Phi_Y\) is identical. If the residual linear derivative is
\(O(\varepsilon_b)\), its contribution along a mode of size \(\varepsilon\)
is \(O(\varepsilon\varepsilon_b)\).

Under strict antagonism,
\[
\Phi_Y-\Phi_Y^\ast
=\eta^\top\mathcal B\xi
=-\chi\eta^\top\mathcal A^\top\xi
=-\chi\xi^\top\mathcal A\eta
=-\chi(\Phi_X-\Phi_X^\ast),
\]
which proves \eqref{eq:antiphase-locking-exact} and the Fourier conclusion.
\end{proof}
\section{Baseline Hopf normal form}
\label{app:baseline-lyapunov}

Using the standard retarded-equation Hopf normal-form construction
\cite{hassard1981}, we compute the cubic coefficient for the baseline example
of Section~\ref{sec:num}. This calculation assumes a simple critical singular
value and does not cover the simultaneous multimode crossing produced by a
repeated critical value. The four hard lags are placed in the equal-split gauge
\[
 \ell_\ast=\frac{\tau_{\rm crit}}4,
 \qquad
 \omega_\ast=\frac1{\sqrt2},
 \qquad
 \tau_{\rm crit}=0.9612039327.
\]
Other allocations with the same total delay are related to this one by the
componentwise shifts described in Section~\ref{sec:consequences}.

Use the orthonormal tangent basis
\[
 e_1=\frac{1}{\sqrt2}(1,-1,0)^\top,
 \qquad
 e_2=\frac{1}{\sqrt6}(1,1,-2)^\top,
\]
and order each tangent block as
\(U_j=(\alpha_j,\beta_j,\xi_j,\eta_j)^\top\), \(j=1,2\).
The two characteristic blocks are
\begin{equation}
 M_j(z)=(z+1)I_4-e^{-z\ell_\ast}
 \begin{pmatrix}
 0&0&0&a_j\\
 0&0&-a_j&0\\
 1&0&0&0\\
 0&1&0&0
 \end{pmatrix},
 \qquad
 a_1=\frac32,
 \quad
 a_2=\frac34.
 \label{eq:baseline-characteristic-blocks}
\end{equation}
The critical mode lies in the first block. A Euclidean-normalized right
vector and a correspondingly normalized left vector are
\begin{equation}
 v_\ast=\sqrt{\frac3{10}}
 \left(1,i,\frac{1-i}{\sqrt3},\frac{1+i}{\sqrt3}\right)^\top,
 \qquad
 w_\ast^\ast M_1'(i\omega_\ast)v_\ast=1,
 \label{eq:baseline-critical-vectors}
\end{equation}
where, to the displayed precision,
\[
\begin{aligned}
 w_\ast={}&(
 0.3612241965+0.0494869053i,
 -0.0494869053+0.3612241965i,\\
 &0.3556862478-0.2699724135i,
 0.2699724135+0.3556862478i)^\top.
\end{aligned}
\]
Implicit differentiation of the critical scalar branch gives
\begin{equation}
 \beta_{\rm H}+i\nu_{\rm H}
 =\left.\frac{dz}{d\tau_\Sigma}\right|_{i\omega_\ast,\tau_{\rm crit}}
 =0.0797591244-0.1534540769i.
 \label{eq:baseline-eigenvalue-derivative}
\end{equation}

In the basis \((e_1,e_2)\), the nonlinear part of the revision field is
obtained from
\begin{equation}
 \mathcal N(a,b;g,h)=
 \begin{pmatrix}
 \dfrac{\sqrt6}{12}(ah+2bg)-a^2g-\dfrac12abh\\[0.7em]
 \dfrac{\sqrt6}{6}ag-\dfrac{\sqrt6}{12}bh-abg-\dfrac12b^2h
 \end{pmatrix}.
 \label{eq:baseline-nonlinear-polynomial}
\end{equation}
The two revision equations contain respectively
\((9/2)\mathcal N(\alpha_1,\alpha_2;\eta_1,\eta_2)\) and
\(-(9/2)\mathcal N(\beta_1,\beta_2;\xi_1,\xi_2)\), with the other population's
coordinates evaluated at delay \(\ell_\ast\); the deployment equations are
linear.

For the standard center-manifold expansion, the quadratic forcing lies wholly
in the noncritical block \(U_2\). Its zero and second harmonics are
\[
\begin{aligned}
 N_{20}&=\left(n,-n,0,0\right)^\top,
 &n&=\frac3{20}(\sqrt6+i\sqrt3),\\
 N_{11}&=\left(\frac{3\sqrt6}{10},
 \frac{3\sqrt6}{10},0,0\right)^\top.
\end{aligned}
\]
Since \(M_2(0)\) and \(M_2(2i\omega_\ast)\) are nonsingular, the corrections
\[
 M_2(2i\omega_\ast)h_{20}=N_{20},
 \qquad
 M_2(0)h_{11}=N_{11}
\]
are uniquely defined and equal to
\[
\begin{aligned}
 h_{20}={}&(
 0.2952741725-0.0869854838i,
 -0.1706761365+0.1067427373i,\\
 &-0.0019274088-0.1777096570i,
 0.0324823533+0.1115930788i)^\top,\\[0.3em]
 h_{11}={}&(
 0.8230285536,
 0.1175755077,
 0.8230285536,
 0.1175755077)^\top.
\end{aligned}
\]
Let \(\mathcal F_2\) and \(\mathcal F_3\) denote the homogeneous quadratic and
cubic parts of the full tangent nonlinearity determined by
\eqref{eq:baseline-nonlinear-polynomial}. In the decomposition
\(U_1\oplus U_2\), embed the critical and noncritical vectors as
\[
 \widehat v_\ast=(v_\ast,0),
 \qquad
 \widehat h_{20}=(0,h_{20}),
 \qquad
 \widehat h_{11}=(0,h_{11}).
\]
Set
\[
\begin{aligned}
 U^{(1)}(t)&=Z\widehat v_\ast e^{i\omega_\ast t}
 +\overline Z\,\overline{\widehat v_\ast}e^{-i\omega_\ast t},\\
 U^{(2)}(t)&=Z^2\widehat h_{20}e^{2i\omega_\ast t}
 +Z\overline Z\,\widehat h_{11}
 +\overline Z^2\overline{\widehat h_{20}}e^{-2i\omega_\ast t}.
\end{aligned}
\]
Writing \(U_t(\theta)=U(t+\theta)\) and evaluating delayed coordinates at \(t-\ell_\ast\), the resonant forcing is
\begin{equation}
 G_{21}
 =\bigl[Z^2\overline Z e^{i\omega_\ast t}\bigr]
 \left\{
 \mathcal F_3\bigl(U^{(1)}_t\bigr)
 +D\mathcal F_2\bigl(U^{(1)}_t\bigr)U^{(2)}_t
 \right\},
 \label{eq:baseline-resonant-vector}
\end{equation}
where the bracket extracts the indicated coefficient. Direct substitution gives,
in the critical block,
\[
 G_{21}=(-0.6766549711-0.1090984654i,
          0.5293334015-1.7809136692i,0,0)^\top.
\]
Hence
\begin{equation}
 c_H=w_\ast^\ast G_{21}
 =-0.9193272749-0.1089995716i.
 \label{eq:baseline-cubic-Hopf-coefficient}
\end{equation}

With \(\Delta\tau=\tau_\Sigma-\tau_{\rm crit}\), the reduced equation is
\begin{equation}
 \dot Z=
 \bigl[i\omega_\ast+(\beta_{\rm H}+i\nu_{\rm H})\Delta\tau\bigr]Z
 +c_HZ|Z|^2
 +o\!\left(|\Delta\tau|\,|Z|+|Z|^3\right).
 \label{eq:baseline-Hopf-normal-form}
\end{equation}
Writing \(Z=re^{i\vartheta}\) gives
\begin{equation}
 \dot r
 =\beta_{\rm H}\Delta\tau\,r
 +\operatorname{Re}c_H\,r^3
 +o\!\left(|\Delta\tau|r+r^3\right).
 \label{eq:baseline-radial-normal-form}
\end{equation}
Thus \(\operatorname{Re}c_H\) is the sign-determining cubic coefficient in the
amplitude convention used here. Standard definitions of the first Lyapunov
coefficient may differ by a positive normalization factor and give the same
criticality classification. Since \(\beta_{\rm H}>0\) and
\(\operatorname{Re}c_H<0\), the baseline Hopf bifurcation is supercritical and
its small postcritical periodic branch is locally orbitally attracting. The deployed-\(x\) component of \(v_\ast\) has
norm \(1/\sqrt5\), so
\begin{equation}
 \max_t\|x(t)-u\|
 =0.26345\sqrt{\Delta\tau}+o(\sqrt{\Delta\tau}),
 \label{eq:baseline-deployed-amplitude-law}
\end{equation}
and
\begin{equation}
 \Omega(\Delta\tau)
 =\omega_\ast-0.1629106770\,\Delta\tau
 +\mathcal O(\Delta\tau^2).
 \label{eq:baseline-frequency-law}
\end{equation}
These predictions are compared with the direct integrations in
Section~\ref{sec:num}. The displayed quantities determine \(c_H\) by finite
linear solves and the adjoint projection.
\section{Numerical procedures and supplementary checks}
\label{app:numerics}
\begingroup
\small

For the same-gain comparison in Figure~\ref{fig:matched-mean-contrast}, the
exploration and implementation rates are one, \(d=1\), \(T=1\), and
\(\Lambda=3.85\), so \(\lambda_X=\lambda_Y=3\sqrt\Lambda\) and
\(\varrho_T=2/3\). The initial perturbation is the real part of the leading
periodic right Floquet vector, normalized to Euclidean size \(10^{-4}\) and
lifted to the first singular mode of the three-strategy system. The initial
state is interpreted as post-release and is identical for both mechanisms.
DOP853 with relative/absolute tolerances \(10^{-11}/10^{-13}\) integrates
the continuous system and each inter-release interval on \([0,200]\),
with resets applied at the prescribed times. Amplitudes are moduli of the
respective left critical projections, normalized at zero; fits use all
post-release samples. They are mechanism-specific modal measurements, not
an identification of the absolute component amplitudes of the two systems.
At monitored solver nodes and post-release samples the smallest portfolio
component is above \(0.3328\), the maximum component deviation is below
\(5.34\times10^{-4}\), and the mass-conservation defect is below
\(2\times10^{-15}\). Tightening the tolerances to
\(2\times10^{-12}/2\times10^{-14}\) changes either fitted exponent by
less than \(5\times10^{-9}\).

The shorter-period check uses \(T=0.5\), \(\Lambda=3.95\),
\(\varrho_T=0.4\), initial size \(10^{-4}\), and \([0,500]\), with
its own leading Floquet vector under the same convention. Its trajectories,
state coordinates, fitted exponents and separate plot are supplied alongside
the main comparison.

For Table~\ref{tab:matched-mean-benchmark} and
Figure~\ref{fig:num-stability-maps}(b), increasing gain grids on
\([0,4\Lambda_0]\) bracket the first detected crossing before Brent
refinement with absolute/relative tolerances \(10^{-12}/10^{-13}\).
The table is checked with 400, 800 and 1600 grid intervals. The curve uses
\(T/d=k/40\), \(k=1,\ldots,80\), with 400 and 800 intervals; the
resulting boundary values differ by less than \(4\times10^{-13}\).
Its value at \(T=0\) is the continuous limit. At \(T=2d\) the reset
is singular but the monodromy matrix remains directly evaluable. These
checks are numerical resolution tests, not certified exclusion of all earlier
crossings. The first-boundary theorem applies only for sufficiently short
periods. At \(T/d=0.1,0.05,0.025\), the normalized shifts
\((4-\Lambda_{\rm cal})/T^2\) are approximately
\(0.333197,0.333299,0.333325\), approaching \(1/3\).

To check the general coefficient away from unit rates, take
\((\mu_X,\mu_Y,\kappa_X,\kappa_Y)=(0.7,1.3,0.9,1.8)\). Then
\(d=1/0.9\), \(\omega_0=1.1009667\), \(\Lambda_0=4.1164931\), and
\(\mathfrak c=0.3833538\). Table~\ref{tab:periodic-unequal-rate}
compares the numerically computed synchronized and phase-averaged boundaries,
using the kernel in \eqref{eq:cal-kernel-expansion}.

\begin{center}
\begin{minipage}{0.94\textwidth}
\captionof{table}{Unequal-rate convergence. Here
\(Q_\bullet=(\Lambda_0-\Lambda_\bullet)/T^2\); both columns tend to
\(\mathfrak c=0.3833538\).}
\label{tab:periodic-unequal-rate}
\centering
\scriptsize
\setlength{\tabcolsep}{4.2pt}
\begin{tabular}{@{}ccccc@{}}
\toprule
\(T/d\) & \(\Lambda_{\rm sync}\) & \(Q_{\rm sync}\) &
\(\Lambda_{\rm av}\) & \(Q_{\rm av}\) \\
\midrule
0.50 & 4.0009630 & 0.374318 & 4.0000137 & 0.377393 \\
0.25 & 4.0870572 & 0.381490 & 4.0870284 & 0.381863 \\
0.10 & 4.1117638 & 0.383075 & 4.1117633 & 0.383115 \\
\bottomrule
\end{tabular}
\end{minipage}
\end{center}
The finite-\(T\) boundaries are distinct, but their normalized displacements
approach the same leading coefficient.

For Figure~\ref{fig:num-time-domain}(b), the full nonlinear hybrid system is
integrated between exact resets by DOP853 with relative/absolute tolerances
\(10^{-11}/10^{-13}\). Here \(T=0.5\), \(\varrho_T=0.4\), all hard lags
vanish, and \(\lambda_X=\lambda_Y=3\sqrt\Lambda\). A size-\(10^{-3}\)
perturbation is formed from the real part of the leading right Floquet vector,
normalized and lifted as in the matched comparison, and monitored by its
normalized left-Floquet projection on \([0,200]\). The fitted exponents are
\(-0.0076319\) and \(0.0083761\), versus Floquet values
\(-0.0076317\) and \(0.0083812\). The regenerated panel and its trajectory
data are included in the supplement.

The continuous DDE runs use the baseline parameters and history
\(p=x=u+10^{-3}(1,-1,0)^\top/\sqrt2\), \(q=y=u\), for \(t\le0\).
A fixed-step method of steps uses classical RK4, piecewise-linear interpolation,
and \(h=5\times10^{-3}\), without projection. Periods and Fourier coefficients
are measured on \([600,800]\). The amplitude fit uses ten runs with
\(0.97\le\tau_\Sigma\le1.20\), continued until successive length-500 windows
differ relatively by less than \(10^{-6}\); its slope is \(0.26298\).

\begin{center}
\begin{minipage}{0.98\textwidth}
\captionof{table}{Fixed-total-delay check. All rows have
\(\tau_\Sigma=1.10\); \(A_x=\max\|x-u\|\), and the last column is the
measured/predicted phase of \(x_1\) relative to \(p_1\).}
\label{tab:delay-split}
\centering
\scriptsize
\setlength{\tabcolsep}{3.0pt}
\begin{tabular}{@{}lcccc@{}}
\toprule
allocation & \((\sigma_X,\sigma_Y;\theta_X,\theta_Y)\) &
\(T_{\rm comp}\) & \(A_x\) & phase \\
\midrule
equal & \((.275,.275;.275,.275)\) & 9.170 & .09795 &
\(-45.2^\circ/-45.2^\circ\) \\
\(X\)-cross-heavy & \((.45,.10;.10,.45)\) & 9.171 & .09794 &
\(-38.3^\circ/-38.3^\circ\) \\
\(Y\)-cross-heavy & \((.10,.45;.45,.10)\) & 9.170 & .09796 &
\(-52.1^\circ/-52.1^\circ\) \\
\bottomrule
\end{tabular}
\end{minipage}
\end{center}
The three runs use the same unshifted history. Their measured phases differ
from \(-\Omega\theta_X-\arctan(\Omega/\kappa_X)\) by less than
\(0.002^\circ\). For \(\chi=2,3\), taking
\(\mathcal B=-\chi\mathcal A^\top\) and
\(\lambda_X=\lambda_Y=9/(2\sqrt\chi)\) preserves \(L=9/4\) and gives
\(\max|\Phi_Y+\chi\Phi_X|<2\times10^{-16}\).

\begin{center}
\begin{minipage}{0.92\textwidth}
\captionof{table}{Step refinement at \(\tau_\Sigma=1.10\), with
\(\mathcal D_h=\max_{t_n,v}|\mathbf1^\top v(t_n)-1|\).}
\label{tab:step-refinement}
\centering
\scriptsize
\setlength{\tabcolsep}{4.0pt}
\begin{tabular}{@{}ccccc@{}}
\toprule
\(h\) & \(T_{\rm comp}\) & \(T_{\rm scalar}\) & \(\max\|x-u\|\) & \(\mathcal D_h\) \\
\midrule
\(10^{-2}\) & 9.170040 & 4.585022 & 0.097950 & \(2.2\times10^{-14}\) \\
\(5\times10^{-3}\)& 9.170036 & 4.585020 & 0.097955 & \(1.2\times10^{-14}\) \\
\(2.5\times10^{-3}\)& 9.170033 & 4.585018 & 0.097958 & \(1.2\times10^{-14}\) \\
\bottomrule
\end{tabular}
\end{minipage}
\end{center}
\endgroup

\end{document}